\documentclass[12pt]{amsart}
\usepackage{latexsym}
\usepackage{amssymb}
\usepackage{amsmath}

\newtheorem{thm}{Theorem}[section]
\newtheorem{lem}[thm]{Lemma}

\newtheorem{prop}[thm]{Proposition}

\newtheorem{prob}[thm]{Problem}

\theoremstyle{remark}
\newtheorem{rem}[thm]{Remark}

\numberwithin{equation}{section}

\newcommand{\beq}{\begin{equation}}
\newcommand{\eeq}{\end{equation}}

\def\real{\hbox{\rm\setbox1=\hbox{I}\copy1\kern-.45\wd1 R}}
\def\prob{\hbox{\rm\setbox1=\hbox{I}\copy1\kern-.45\wd1 P}}
\def\natural{\hbox{\rm\setbox1=\hbox{I}\copy1\kern-.45\wd1 N}}
\def\integers{\hbox{\rm\setbox1=\hbox{-}\copy1\kern-1.2\wd1 Z}}

\newcommand{\B}{\mbox{$\mathcal B$}}

\begin{document}
\title{Joint coboundaries}
\author{Terry Adams and Joseph Rosenblatt}
\date{July, 2016}

\begin{abstract}  We ask under what conditions on the function $f$, and a set of maps $\mathcal T$, it is the case that $f$ is a coboundary for some map in $\mathcal T$.  We also consider for a function $f$, and a set of maps $\mathcal T$, when we have $f$ being a coboundary for all the maps in $\mathcal T$.
\end{abstract}

\maketitle

\section {\bf Introduction}\label{intro}

Let $T:\mathcal F \to \mathcal F$ be a continuous linear mapping where $\mathcal F$ is a Banach space (of functions). Suppose $\mathcal E$ is another Banach space which is a subspace of $\mathcal F$ with $T(\mathcal E) \subset \mathcal E$.  Given $f\in \mathcal E$,
we say that $f$ is a {\em $\tau$-coboundary} with {\em transfer
function} $h \in \mathcal F$ if $f = h - T(h)$.

We are interested in being
able to recognize when $f$ is a $T$-coboundary for different Banach
spaces $\mathcal F$ and $\mathcal E$, and various mappings $T$.  In particular, given
a class of continuous linear maps $\mathcal T$, we study the following:

\begin{enumerate}
\item Given $f$, how large is the set $\mathcal S \subset \mathcal T$ such that $f$ is a $T$-coboundary for every $T \in \mathcal S$?
\item Given $f$, does there always exist some $T\in \mathcal T$ for which $f$ is a $T$-coboundary?
\item  For which $\mathcal S \subset \mathcal T$ does there exist a function $f\in \mathcal E$ which is a $T$-coboundary for every $T\in \mathcal S$?
\end{enumerate}

It is important that we be able to vary the
ambient space $\mathcal F$ containing the transfer function.  For example,
consider $\mathcal E = C(\mathbb T)$ with $\mathbb T$ being the circle group,
and let $\mathcal T$ consist of all
rotations $T(h) = h\circ \tau_\alpha$ where $h \in \mathcal F$, for some
class of functions $\mathcal F$ on $\mathbb T$.  Here $\tau_\alpha$ is the rotation given by
$\tau_\alpha(\beta) = \alpha\beta$ for all $\alpha,\beta \in \mathbb T$.
If we take $f$ to be a
trigonometric polynomial, then for any rotation $\tau_\alpha$ there
is a transfer function $h$ which is itself a trigonometric
polynomial for which $f = h - h \circ\tau_\alpha$.  However, if $f
\in C(\mathbb T)$ and is not a trigonometric polynomial, then with
$\mathcal F = L_1(\mathbb T)$, there is only a first category set of $\alpha$ (possibly empty)
for which the coboundary equation can hold.  It is an open question
whether or not the same thing is true when $\mathcal F$ is allowed to consist
of any Lebesgue measurable function.  See Baggett~\cite{Baggett} and Baggett,
Medina, and Merill~\cite{BMM}.

We want to consider the coboundary questions in a common context that arises in ergodic theory.  Take a standard Lebesgue probability space $(X,\mathcal B,p)$ and a {\em map}
$\tau:X\to X$ i.e. a measure-preserving invertible transformation of $(X,\mathcal B,p)$.
Given $r, 1 \le r \le \infty$ and $1 \le s \le \infty$, we study the class of functions  $f\in L_r(X)$ such that $f$ is a {\em $\tau$-coboundary} with transfer function in $L_s(X)$ i.e.
there exists $h \in L_s(X)$ such that $f = h - h\circ\tau$.

For the most part, we consider the coboundary problem with $r = s$, but there are interesting issues if we allow $s < r$ too.  
When $s < r$ and
$f\in L_r(X)$, then we could possibly solve the coboundary equation with  $\tau$-transfer functions in $L_s(X)$, when we could not solve this equation with transfer function in $L_r(X)$.    The special case where $s=0$ (i.e., the transfer functions are only assumed to be measurable) is 
a very interesting case, just as it is in the problem of Baggett~\cite{Baggett,BMM}.  This type of problem for ergodic maps is discussed some in Section~\ref{gencases}

There have been many articles dealing with coboundaries and their characterizations in various contexts.  In the standard
case of maps, a function $f\in L_r(X)$ is a $\tau$-coboundary with transfer function in $L_r(X)$ if and only if the sums $S_n^\tau f = \sum\limits_{k=0}^{n-1} f\circ \tau^k$,
with $n \ge 1$, are uniformly $L_r$-norm bounded i.e. $\sup\limits_{n \ge 1} \|S_n^\tau f\|_r < \infty$.  See Lin and Sine~\cite{LS} for these
results for all $r, 1 \le r \le \infty$.  Note: the result when $1 < r < \infty$ is quite a bit older, going back to at least
Browder~\cite{B}.

On some level, it is easy to tell if $f$ is a $\tau$-coboundary.  We just compute $\|S_n^\tau f\|_r$ and see if it
is bounded as $n$ varies.  If not, then $f$ is not a $\tau$-coboundary.  If these norms are uniformly bounded,
at least, then  $f$ is a $\tau$-coboundary.  But this is somewhat misleading.  Indeed, is it
so clear that there cannot exist a non-zero function that is a coboundary for all (ergodic) maps?  Also, is it
possible that there is a mean-zero function which is not a coboundary for any (ergodic) map?
This article looks at a variety of questions of this type
related to the coboundary equation . For example, in what circumstances is a
given function $f$ a coboundary for a set $\mathcal T$ of maps (i.e., $\mathcal T$ has a non-trivial common
coboundary)?  If $\tau$ and $\sigma$
commute, then $\mathcal T = \{\tau,\sigma\}$ has a common coboundary.  The same is true in somewhat more general cases (e.g., $\sigma\tau\sigma^{-1}$ is a power of $\tau$).  It is natural to ask what properties of $\sigma$
and $\tau$ are necessary and sufficient for there to be a non-trivial common coboundary.  Although we do not have an explicit example of a pair of maps for which there are no non-trivial common coboundaries, the generic pair of maps have no common coboundaries.  This issue is discussed in Section~\ref{nojoint}.
Additionally, even when $\mathcal T$ consists of mutually commuting maps, it may be that there are no common coboundaries
if $\tau$ varies over all of $\mathcal T$.  We can ask, for instance, when there is a $\tau$-coboundary which is
a $\tau^k$-coboundary for all $k \in \mathbb Z$.  This question is discussed in Section~\ref{powers}.

\section{Role of coboundaries in rate of norm convergence}

There is an aspect of coboundaries that should be kept in mind.
The set of $\tau$-coboundaries  $\mathcal D_r^\tau = \{h - h\circ \tau: h \in L_s(X)\}$ is always a linear space.
The standard case is when $r=s,
1 \le r < \infty$.  Let $\mathcal I^\tau$ denote the invariant functions $\{f\in L_r(X): f\circ \tau = f \}$.
We always have $\mathcal D_r^\tau + \mathcal I^\tau$ is $L_r$-norm dense in $L_r(X)$.   When $\tau$ is ergodic,
$\mathcal I^\tau$ consists just of the constant functions, and $\mathcal D^\tau$ is $L_r$-norm dense in the mean-zero functions
in $L_r(X)$.
At the same time, the category statement in Proposition~\ref{mostfcnsnotcob} shows that the generic function is not a coboundary, a well-known fact at least for transfer functions in $L_1(X)$.

Let $A_n^\tau f$ be the standard ergodic average $\frac 1n\sum\limits_{k=0}^{n-1} f\circ \tau^k$.  It is an interesting aspect of the approximation of mean-zero functions by coboundaries that this is what completely determines the rate that $\|A_n^\tau f\|_r$ tends to zero.  First, take $f \in L_r(X)$
and $h \in L_r(X)$ with $1 \le r < \infty$.  Then
\[\|A_n^\tau f\|_r \le \|A_n^\tau (f - (h-h\circ \tau))\|_r + 2\|h\|_r/n.\]
So
\[\|A_n^\tau f\|_r\le \inf_{\|h\|_r\le L}\|f - (h-h\circ \tau)\|_r + 2L/n.\]
Hence, with $\tau$ ergodic, the approximation of mean-zero $f$ by
coboundaries gives an upper-bound for the rate that $\|A_n^\tau f\|_r$ goes to zero.

On the other hand, let $S_n^\tau f =\sum\limits_{k=0}^{n-1} f\circ \tau^k$.
We can express $f - A_n^\tau f = h - h \circ \tau$ with $h = \frac
1n\sum\limits_{k=1}^n S_k^\tau f$.  Indeed,
\begin{eqnarray*}
f - A_n^\tau f &=&
\frac 1n \sum\limits_{k=0}^{n-1} (f - f\circ \tau^k)\\
&=&\frac 1n \sum\limits_{k=0}^{n-1} (S_k^\tau f - S_k^\tau f\circ \tau)\\
& = & \frac 1n \sum\limits_{k=1}^n S_k^\tau f -
\left (\frac 1n \sum\limits_{k=1}^n S_k^\tau f\right )\circ \tau.
\end{eqnarray*}
Here $\|h\|_r\le \frac 1n\sum\limits_{k=1}^n k\|f\|_r\le n\|f\|_r$. So
\[ \|A_n^\tau f\|_r \ge  \inf_{\|h\|_r\le n\|f\|_r}\|f - (h-h\circ \tau)\|_r.\]
This gives a lower-bound for $ \|A_n^\tau f\|_r$ in terms of how well $f$ is
approximated by $\tau$-coboundaries.

These two bounds are not exactly comparable.  Moreover, of course the value of the norm $\|h\|_r$ of the transfer function plays a role too.  But nonetheless, these estimates show one important traditional role for coboundaries.

The rate that $\|A_n^\tau f\|_r$ tends to zero, is of course directly connected to the size of $S_n^\tau f$.  There are interesting, good results relating to the growth of  the norms of cocycles.  There is a large literature on this subject.  Here are some articles to look at for results: Derriennic~\cite{Derriennic}, Derriennic and Lin~\cite{DL}, Gomilko, Hasse, and Tomilov~\cite{GHT}, and Rosenblatt~\cite{R2014}.  These articles give some evidence that the following holds (or at least it would be interesting if it did not hold).

\medskip

\noindent {\bf Conjecture}: Take $\tau$ ergodic, $1\le r \le \infty$, and
any function $\phi(n) = o(n)$.  Then there is (and one can explicitly construct) a mean-zero $f\in L_\infty(X)$ such that $\|S_n^\tau f\|_r \sim \phi(n)$ as $n\to \infty$.

\medskip

\section{Algebra gives common coboundaries}\label{commute}

The most basic example is when $\tau$ and $\sigma$ commute.  Then for any $h \in L_r(X)$, we have
$(h - h\circ\tau) - (h - h\circ\tau)\circ \sigma = (h - h\circ\sigma) - (h - h\circ\sigma)\circ \tau$.  Hence,
$f = (h - h\circ\tau) - (h - h\circ\tau)\circ \sigma$ is both a $\sigma$-coboundary and a $\tau$-coboundary.

This can be generalized to include algebraic conditions on the group that $\sigma$ and $\tau$ generate.  For example,
suppose $\tau\sigma = \sigma\tau^2$.  Then
\begin{eqnarray*}
(h - h\circ\tau) - (h - h\circ\tau)\circ \sigma &=& (h - h\circ\tau) - (h\circ \sigma - h\circ\tau\circ \sigma)\\
&=& (h - h\circ \tau) - (h\circ\sigma - h\circ \sigma\circ \tau^2)\\
&=&(h - h\circ\sigma -h\circ\sigma\circ\tau) - (h - h\circ\sigma -h\circ\sigma\circ\tau)\circ\tau.
\end{eqnarray*}
Hence, the $\sigma$- coboundary $(h - h\circ\tau) - (h - h\circ\tau)\circ \sigma$ is also a $\tau$-coboundary.

Also, if $\sigma=\tau^n$ for some $n$, then any $\sigma$-coboundary is a $\tau$-coboundary by a collapsing sum argument. For
example, $h - h\circ \tau^2 = h + h\circ \tau - h\circ \tau - h\circ \tau^2 = (h + h\circ \tau)-(h+h\circ \tau)\circ \tau$. But
more generally, if $n$ is allowed to be an integrable function on $X$, then this remains true.  See Dajani~\cite{D} for
the details on this.

\begin{rem}  There are also a number of related results in both Kornfeld~\cite{K} and Kornfeld and Losert~\cite{KL}.
These results are in the direction of finding algebraic conditions on $\mathcal T$ that guarantee there are non-trivial
common coboundaries.  But certainly we are far from having a very general algebraic condition on a set of maps (even just two maps) which guarantee
they do have common coboundaries (in some transfer class), let alone a characterization of this.
\end{rem}

\section{Coboundaries and automatic continuity}\label{automatic}

In this section,
we consider first a related issue: when can we represent all mean-zero functions as sums of coboundaries.  This is
very closely related to the existence of unique invariant means on $L_\infty(X)$.  So for this phenomenon to occur
the maps must be quite far from having any properties such as in Section~\ref{commute}.  So this suggests at least
 the possibility of having explicit maps with no common coboundaries.

For example, consider a finite set $\tau_1,\dots,\tau_m$ of maps for which the integral with respect
to $p$ is the unique invariant mean on $L_\infty(X)$.  In Rosenblatt~\cite{R1991}, it is  shown that
this is equivalent to automatic continuity of $\{\tau_1,\dots,\tau_m\}$-invariant linear functionals of the Lebesgue
spaces $L_r(X), 1 < r \le \infty$.  In particular, this uniqueness property implies that for every mean-zero $f \in L_r(X)$, there exist $h_k,k=1,\dots,m$
in $L_2(X)$ such that $f = \sum\limits_{k=1}^m h_k - h_k\circ \tau_k$ i.e. $f$ is a sum of $\tau_k$-coboundaries.
See also Lind~\cite{Lind} where one obtains this type of representation with commuting maps, by allowing the transfer functions to be measurable.  This clearly shows the impact of the class of the transfer function since this could not happen with commuting maps if the transfer functions were in $L_r(X), r \ge 1$.

From Rosenblatt~\cite{R1981}, using Rosenblatt~\cite{R1991}, we see that we have this particular example.
The toral automorphisms correspond to the unimodular group  (i.e., the integer entry $n\times n$ matrices with determinant
$1$ or $-1$).  Generally, $SL(n,\mathbb Z)$ is the subgroup of matrices whose determinant is $1$, a subgroup of index two in
the unimodular group.

\begin{prop}\label{autoeg} Consider the probability space $(\mathbb T^2,\mathcal B,p)$ where $p$ is the usual Lebesgue measure
on the two-torus $\mathbb T^2$, and $\mathcal B$ is the Lebesgue measurable sets.  Let $1 < r \le \infty$.  Let $\alpha_1(x,y) = (y,x)$ and $\alpha_2(x,y)
= (x,xy)$ for all $(x,y) \in \mathbb T^2$.  Then for all mean-zero $f\in L_r(\mathbb T^2)$, there exist $h_1,h_2 \in L_r(\mathbb T^2)$
such that $f = h_1 - h_1\circ \alpha_1+ h_2 - h_2\circ\alpha_2$.
\end{prop}
\begin{proof} The uniqueness of the $\{\alpha_1,\alpha_2\}$-invariant mean implies that for every $f \in L_r(\mathbb T^2)$, there exists $\beta_k,k=1,\dots,m$
in the countable group generated by $\{\alpha_1,\alpha_2\}$  and $h_k, k=1,\dots,m$ in $L_r(\mathbb T^2)$ such that $f = \sum\limits_{k=1}^m h_k - h_k\circ \beta_k$.
But actually each $f_k - f_k\circ \beta_k$ is of the form either $h-h\circ\alpha_1$ or $h - h\circ\alpha_2$.  Indeed, suppose $\beta = \alpha_{i_1}^{e_1}\dots\alpha_{i_n}^{e_n}$
where each $\alpha_{i_j}$ is either $\alpha_1$ or $\alpha_2$ and each $e_j$ is either $1$ or $-1$.   Let
$\gamma_k = \alpha_{i_1}^{e_1}\dots\alpha_{i_k}^{e_k}$ for $k \ge 1$, and $\gamma_0 = 1$.
Then $h - h\circ \beta = \sum\limits_{k=1}^m h\circ \gamma_{k-1} - h\circ \gamma_k = \sum\limits_{k=1}^m h\circ \gamma_{k-1} - (h\circ \gamma_{k-1}) \circ \alpha_{i_k}^{e_k}$.
Notice also that $h - h\circ \alpha^{-1} = (-h\circ \alpha^{-1}) - (-h\circ \alpha^{-1})\circ \alpha$.  Thus, the terms
$h\circ \gamma_{k-1} - (h\circ \gamma_{k-1}) \circ \alpha_{i_k}^{e_k}$ are each of the form either $H - H\circ \alpha_1$ or $H - H\circ \alpha_2$ for some $H \in L_r(\mathbb T^2)$.
Hence, $h - h\circ \beta$ is of the form $h_1 - h_1 \circ \alpha_1+ h_2 - h_2\circ \alpha_2$.
\end{proof}

\begin{rem}\label{notjointerg} It is also probably the case that there are two ergodic mappings $\tau_1,\tau_2$ of $(X,\beta,p)$
on $\mathbb T^2$ with the same property as in Proposition~\ref{autoeg}.  But an example is not known at this time.
\end{rem}

This representation gives the following using the same notation.

\begin{prop}\label{notjoint}  There is a function $f \in L_r(\mathbb T)$ which is an $\alpha_1$-coboundary but not an $\alpha_2$-coboundary.
\end{prop}
\begin{proof}  Consider the functions $f = h_1 - h_1\circ \alpha_1$.  If these are always $\alpha_2$-coboundaries, then Proposition~\ref{autoeg} shows that the mean-zero functions
in $L_r(\mathbb T^2)$ would be exactly the $\alpha_2$-coboundaries.  This is impossible because the coboundaries with respect to a single map are always a first category
subspace.
\end{proof}

We are interested in examples of pairs of maps that do not have a common coboundary in some Lebesgue space.  The two maps $\alpha_1$ and $\alpha_2$ above seemed at first
as good candidates to prove this.  However, it turns out to actually be otherwise.  Now, we have seen that some algebraic condition will ensure
this, but the maps above do not satisfy any of those conditions.  In fact, $\alpha_1$ and $\alpha_2$ generate the unimodular group; see Trott~\cite{Trott}.  However,
Proposition~\ref{autoeg} can be used to prove the following.

\begin{prop} The maps $\alpha_1$ and $\alpha_2$ have a non-trivial common coboundary in  $L_2(\mathbb T^2)$ with transfer functions in $L_2(\mathbb T^2)$.
\end{prop}
\begin{proof} Let $\mathcal I_\alpha = \{f\in L_2(\mathbb T^2): f = f\circ \alpha\}$ and $\mathcal D_\alpha = \{g - g\circ \alpha: f\in L_2^0(\mathbb T^2)\}$.  Then $\mathcal I_\alpha^\perp$
is the $L_2$-norm closure $\overline {\mathcal D_{\alpha}}$ of $\mathcal D_\alpha$.  Also, $L_2(\mathbb T^2) = \mathcal I_{\alpha} + \overline {\mathcal D_{\alpha}}$.  Consider the mapping
$A:\mathcal I_{\alpha_1}^\perp\times \mathcal I_{\alpha_2}^\perp\to L_2^0(\mathbb T^2)$ by $A(h,g) = h-h\circ \alpha_1+g -g \circ \alpha_2$.   Then $A$
would be a continuous linear map.  It is onto from Proposition~\ref{autoeg}.

Assume that
$\alpha_1$ and $\alpha_2$ have no non-trivial common coboundaries in $L_2(\mathbb T^2)$ with transfer functions in $L_2(\mathbb T^2)$. We claim that $A$ is then one-to-one.
Suppose $(h_0,g_0)\in \mathcal I_{\alpha_1}^\perp\times \mathcal I_{\alpha_2}^\perp$ such that $A(h_0,g_0) = 0$.  That is, $ h_0 - h_0\circ\alpha_1 = -g_0 - (-g_0)\circ \alpha_2$.  Thus, $H = h_0- h_0\circ\alpha_1$
is a common coboundary for $\alpha_1$ and $\alpha_2$, and so  $H$ must be zero.  But then  $h_0 \in \mathcal I_{\alpha_1}$ and $g_0 \in \mathcal I_{\alpha_2}$.  Since $h_0\in \mathcal I_{\alpha_1}^\perp$, we have $h_0 = 0$,
and since $g_0\in  \mathcal I_{\alpha_2}^\perp$, we have $g_0 =0$.  Hence, $(h_0,g_0)$ zero.

So if $\alpha_1$ and $\alpha_2$ have no non-trivial common coboundaries, then $A$ is an isomorphism.  Hence, if $(h_n)$ is in $L_2^0(\mathbb T^2)$ and $h_n-h_n\circ \alpha_2$
converges in $L_2$-norm to $H$, we would have $(h_n)$ converging in $L_2$-norm too and so there would be $h \in L_2^0(\mathbb T^2)$ such that $H = h-h\circ \alpha_2$.
That is, $\mathcal D_{\alpha_2}$ is $L_2$-norm closed.   We claim this is not the case.

We can see that $\mathcal D_{\alpha_2}$ is not $L_2$-norm closed as follows.  It is clear that $\mathcal I_{\alpha_2}$ is all $f\in L_2(\mathbb T^2)$ such that for a.e. $x$, the function $f(x,y)$ is constant a.e. in $y$.
But then $\mathcal I_{\alpha_2}^\perp$ consists of all functions $f \in L_2(\mathbb T^2)$ such that $f(x,y)$ is mean-zero in $y$ for a.e. $x$.  Now consider a mean-zero function $\phi\in L_2(\mathbb T)$ which is not a $\tau_x$-coboundary
for any $x\in \mathbb T$ of infinite order; there are such functions that are in $C(\mathbb T)$.  See Remark~\ref{rot}.
Now let $f(x,y) = \phi(y)$.  Then $f\in \mathcal I_{\alpha_2}^\perp$ i.e. it is in the $L_2$-norm
closure of $\mathcal D_{\alpha_2}$.  But if $f\in \mathcal D_{\alpha_2}$ too, then $f(x,y) = \phi(y) = h(x,y) - h\circ \alpha_2(x,y) = h(x,y) - h(x,xy)$ for some $h \in L_2(\mathbb T^2)$.  Hence, for a.e. $x$, $h(x,y)$ is
an $L_2(\mathbb T)$ function in $y$, and $\phi$ would be a $\tau_x$-coboundary with transfer function $H(y) = h(x,y) \in L_2(\mathbb T)$.  However, a.e. $x$ is infinite order too, and $\phi$ cannot have this property with
respect to any $x$ of infinite order.
\end{proof}

\begin{rem}\label{extend}  We can extend this result using duals in the Lebesgue spaces.  This would give, for any $1 \le r < \infty$, a non-trivial common coboundary $H\in L_r(\mathbb T^2)$
with transfer functions in $L_r(\mathbb T^2)$ also.  The argument does not work in $L_\infty(\mathcal T^2)$, so it is not clear if there is a non-trivial bounded common coboundary for $\alpha_1$ and
$\alpha_2$ with transfer functions in $L_\infty(\mathbb T^2)$.
\end{rem}

\begin{rem}\label{anothereg} The method above can give examples more easily if we switch to the generators $\alpha_1$ and $\alpha_3 =\alpha_1\circ \alpha_2$ of $SL(2,\mathbb Z)$.  Then $\alpha_3$ is
ergodic so it is well-known that the coboundaries $\mathcal D_{\alpha_3}$ is not $L_2$-norm closed.  There may be a way of using the existence of non-trivial common coboundaries for one pair of
generators of the unimodular group to get the same thing for another pair of generators, but it is not clear at this time how to do this.
\end{rem}

It is not obvious what can be said if one needs to use more than two toral automorphisms to write a given mean-zero function as a sum of coboundaries.  But at least we can ask how to analyze other pairs of toral automorphisms.
The same idea as above would work if the maps $\alpha_i,i=1,2$ are such that invariance with respect to these maps implies automatic 
continuity of the linear form.  One expects this to be connected to the size 
the subgroup generated by $\alpha_1$ and $\alpha_2$.  But it seems that the random pair $\alpha_i,i=1,2$ generates a thin group and so would be  far from being obviously large enough for the arguments above to work.

\section{Generic case for no common coboundaries}\label{nojoint}

Suppose $\tau$ and $\sigma$ are ergodic.  To say there are not
common coboundaries with transfer function in $L_r(X)$ would mean
that if $f = h-h\circ\tau = g - g\circ \sigma$, for some $h, g \in
L_r(X)$,  then $f = 0$.  That is, if $h-h\circ\tau + (-g) -
(-g)\circ \sigma = 0$, and $h,g \in L_r(X)$ are mean-zero, then $h=
g = 0$.  Let $L_r^0(X)$ denote the mean-zero functions in $L_r(X)$.
We see in this situation that the map $S(h,g) = h-h\circ\tau + g - g\circ \sigma$ from
$L_r^0(X)\bigoplus L_r^0(X)$ to $L_r^0(X)$ is one-to-one.

\begin{prop}\label{dual}  For $1 < p \le \infty$, the map $S$ is one-to-one if and only if, with $q$ the dual index to $p$,
for all $\epsilon > 0$, and  $K_1,K_2 \in L_q^0(X)$, there exists $H
\in L_q(X)$ such that $\|H - H \circ \tau- K_1\|_q \le \epsilon$ and
$\|H - H \circ \sigma - K_2\|_q \le \epsilon$.  Also, $S$ is
one-to-one on $L_1(X)$ if for $K_1,K_2 \in L_\infty^0(X)$, there
exists $H \in L_\infty(X)$ such that $H - H \circ \tau$ approximates
$K_1$ in the weak* topology and simultaneously $H - H \circ \sigma$
approximates $K_2$ in the weak* topology.
\end{prop}
\begin{proof} We have $S$ one-to-one if and only whenever $\langle S(h,g),K\rangle = 0$ for all $K\in L_q(X)$, then $h = g = 0$.
In the case that $1 \le p < \infty$, this means that the dual
operator $S^*$ has dense range. But the dual operator $S^*$ maps
$L_q(X) \to L_q^0(X)\bigoplus L_q^0(X)$ by $S^*(H) = (H -
H\circ\tau^{-1}, h - H\circ\sigma^{-1})$.  In the case that $p =
\infty$, we recognize that $S$ is a dual operator to this same dual
form, and so a similar argument can be used.  The case of $p = 1$ is
the same using the weak* topology in place of the norm topologies.

But also $h-h\circ\tau + g - g\circ \sigma = (-h\circ\tau) - (-h\circ\tau)\circ\tau^{-1} + (-g\circ\sigma) - (-g\circ\sigma)\circ\sigma^{-1}$, so $S$ is one-to-one if and
only if the similar mapping $T(h,g) = h- h\circ\tau^{-1} + g- g \circ\sigma^{-1}$ is one-to-one.  The dual argument above applied to $T$ gives the
result.
\end{proof}

\begin{rem} We would like to get a result for the biggest space $L_1(X)$, and one that allows measurable transfer functions.  But the argument below
does not carry that far at this time.
\end{rem}

The dual property shows that $\tau$ and $\sigma$ have no common
coboundaries if and only if we can carry out the following two
separate approximations.  We state this for $1 \le q < \infty$, but
it works the same for $q = \infty$ using the weak* topology.  The
approximations are this: given $\epsilon > 0$, and $K_i \in
L_q^0(X),i=1,2$ there exists $H_i\in L_q^0(X)$ such that a) $\|H_1 -
H_1 \circ \tau - K_1\|_q \le \epsilon$ and $\|H_1 - H_1 \circ \sigma
\|_q \le \epsilon$, and
b)  $\|H_2 - H_2 \circ \sigma - K_2 \|_q \le
\epsilon$ and $\|H_2 - H_2 \circ \tau\|_q \le \epsilon$.
These would be necessary properties, and if they hold then taking $H
= H_1+H_2$ gives the needed function in Proposition~\ref{dual}.

An additional simplification is useful.  We can reduce this joint
approximation property to the case where the functions $K_i$ are
mean-zero and take only the values $\pm 1$. Such functions
span a dense subspace of $L_q^0(X)$ (in the weak* topology in the
case that $q = \infty$). But then linearity of the approximation
process here allows us to add approximate solutions and get the
general approximation from that for these simpler functions $K_i$.

These remarks set up the following existence result. It is implicitly constructive, but does not give explicit, easily described examples.
\begin{prop}  The generic pair of maps $(\sigma,\tau)$ have no common
boundaries with transfer functions in $L_r^0(X), 1 \le p \le \infty$.
\end{prop}
\begin{proof}  Assume first that $1 < p \le \infty$.
Let $\mathcal D$ be a countable set of functions taking only the
values $\pm 1$ which span a dense subspace of $L_q^0(X)$.
Consider the set of pairs $(\sigma,\tau)$ such that for all $m \ge
1$ and $K \in \mathcal D$, there exists $H \in L_q(X)$ such that
$\|H - H\circ \sigma\|_q < 1/m$ and $\|H - H\circ \tau - K\|_q <
1/m$.  This set $\mathcal E_1$ is a countable intersection of open
sets in $\mathcal T\times \mathcal T$.  We will show that this is a
dense set in $\mathcal T\times \mathcal T$.  We can reverse the
roles of $\sigma$ and $\tau$ to prove the same thing for the
corresponding set $\mathcal E_2$.  But then the intersection
$\mathcal  E_1\cap \mathcal E_2$ is a dense $\mathcal G_\delta$ set
for pairs $(\sigma,\tau)$ for which both a) and b) hold.  By
Proposition~\ref{dual} and the discussion above, this proves the
result.

Now fix an ergodic, rank one map $\sigma$. We show that there is a
dense $G_\delta$ set $\mathcal E_1(\sigma)$ consisting of maps
$\tau$ such for all $\epsilon > 0$ and $K \in \mathcal D$, there
exists $H\in L_q(X)$ such that $\|H - H\circ\sigma\|_q  < \epsilon$
and $\|H - H \circ\tau - K\|_q < \epsilon$.  If we take a countable
dense set $\Sigma$ of rank, one maps, then the set
$\bigcap\limits_{\sigma\in \Sigma}\mathcal E_1(\sigma)$ would again
be a dense $\mathcal G_\delta$ set. But then using the density of
this set and the choice of $\Sigma$ being dense, this shows that
$\mathcal E_1$ is a dense set.

So fix an ergodic, rank one map $\sigma$ and a countable set
$\mathcal D$ as above. Suppose $l \ge 1$. Let $\mathcal D(l,\sigma)$
be all $H \in L_q(X)$ such that $\|H - H\circ\sigma\|_q < 1/l$.
Consider the following set $\mathcal G = \bigcap\limits_{K \in
\mathcal D}\bigcap\limits_{l=1}^\infty \bigcup\limits_{H \in
\mathcal D(l,\sigma)} \{\tau \in \mathcal T: \|H - H \circ\tau -
K\|_q < 1/l\}$.  The union here is a union of open sets in the weak
topology on $\mathcal T$.  So $\mathcal G$ is a $G_\delta$ set in
$\mathcal T$. We claim it is dense because each $\bigcup\limits_{H
\in \mathcal D(l,\sigma)} \{\tau \in \mathcal T: \|H - H \circ\tau -
K\|_q < 1/l\}$ is dense.

To see this density, it suffices to show that for any $\tau_0$,
$\epsilon > 0$ and $K \in \mathcal D$, there exists $\tau$ close to
$\tau_0$ in the weak topology, and there exists $H \in L_q(X)$, such
that $\|H - H\circ\sigma\|_q \le \epsilon$ and $\|H - H \circ\tau -
K\|_q \le \epsilon$.

Since $\sigma$ is rank one, there is a Rokhlin tower $\mathcal R$
with levels $R_j,j=1,\dots,N$ for $\sigma$ whose levels can be used
to approximate any previously chosen measurable partition of $X$.  We
need only increase $N$ and decrease $\epsilon = 1 -
p(\bigcup\limits_{j=1}^N R_j)$ to achieve this approximation to any
desired degree of accuracy.

To get close to $\tau_0$, we can proceed as follows.  Take a weak
neighborhood $W$ of $\tau_0$.  Consider pairwise disjoint $(P_l)$
and pairwise disjoint $(Q_l)$ with $p(P_l) = p(Q_l)$ for all.  Assume
that for each $l$, both $P_l$ and $Q_l$ are equal to unions
of (the same number) of levels $R_j$.   There is such a choice of
sets so that both $(P_l)$ and $(Q_l))$ are close to being partitions
of $X$, and  so that if $\tau$ maps $P_l$ to $Q_l$ for all $l$,
then $\tau \in W$. Indeed, we may assume that if $\tau$ just has
$p(\tau(P_l)\Delta Q_l)$ sufficiently small for all $l$,
then that suffices to guarantee that  $\tau \in W$.  By an
additional approximation if needed, we may also assume without loss of
generality that $K$ is
a constant $a_l$ on the sets $P_l$ for all, and the norm of $K$
restricted to the complement of $\bigcup\limits_{j=1}^J P_l$ is small. Note that by choice of the
functions $K$ being used, all $a_l$ are either $\pm 1$. Here
it is worth observing that $\tau_0$ does not necessarily (and probably
does not) preserve the partition of $P_l$ and $Q_l$ into the levels
$R_j$.

We partition each $R_j$ into sets $R_{(j,i)}, i=1,\dots,M$ of equal
measure. Here $p(R_j) = (1-\epsilon)/N$ for all $j$, and
$p(R_{(j,i)}) = (1-\epsilon)/NM$ for each $(j,i)$.  We assume that
these sets are chosen so that for $j=1,\dots,N-1$ and $i=1,\dots,M$, $\sigma(R_{(j,i)}) =
R_{(j+1,i)}$.  We take as our function $H$ a function which has the
value $i$ on $\bigcup\limits_{j=1}^J R_{(j,i)}$.  It is important
that for a sufficiently tall Rokhlin tower, $H$ is close to being
$\sigma$-invariant.  Actually, the error here for the
$\sigma$-invariance is  controlled by $\|H - H\circ \sigma\|_q^q \le
2\sum\limits_{i=1}^M i^q/MN \le 2M^q/N$.

Now we define $\tau$ as follows.  Here is the basic idea, and first step
 in defining $\tau$.  Take $\tau$ to almost map
$P_1$ to $Q_1$.  But also, for $x
\in P_1$, we want $H(x) - H(\tau (x)) = a_1$.  Now $P_l$ is a union
of some $R_{(j_1,i)}, j_1 \in \mathcal J_1,i=1,\dots,M$, and $Q_l$ is a union of
some
$R_{(j_2,i)}, j_2 \in \mathcal J_2, i=1,\dots, M$. Here $J_1$ and $J_2$ have the
same number of terms because $p(P_1) = p(Q_1)$. Choose any bijection $t$ of $J_1$ with $J_2$.
If $a_1 = -1$, we
take $\tau$ to be a map such that for all $j_1\in J_1$,
we have $\tau(R_{(j_1,i)}) = R_{(t(j_1),i+1)}$, except that for $i = M$
we leave $\tau$ undefined on all $R_{(j_1,M)}$.  Note that this also leaves $R_{(j_1,1)}$ with
$j_1\in \mathcal J_1$ not
in the defined range of $\tau$ for now.
If $a_1 = 1$, we
take $\tau$ to be measure-preserving such that for all $j_1 \in J_1$,
we have $\tau(R_{(j_1,i)}) = R_{(t(j_1),i-1)}$, except that for $i = 1$
we leave $\tau$ undefined.  This leaves $R_{(j,M)}$ with
$j\in \mathcal J_1$ not
in the defined range of $\tau$ for now.
We continue this process in the same fashion through
all of the pairs $P_l, Q_l$ depending on the values $a_l = \pm 1$.

Because $K$ is mean-zero, there are as many
exceptions in the above where $i=M$ as where $i=1$.  That is, taking
into account the number of $j$ such that $R_{(j,i)} \subset P_l$, we
have the same number of times that $\tau$ is not defined on $R_{(a,M)}$
as where $\tau$ does not have $R_{(b,M)}$ in its range.  We take any
one-to-one correspondence of these sets,
and take $\tau$ to map $R_{(a,M)}$ to $R_{(b,M)}$ in a
measure-preserving fashion.  Since $H = M$ on these sets, we would
have now $\tau$ defined and $H - H\circ \tau = 0$ on all the sets $R_{(a,M)}$.
We can carry out the same process for the cases where
$\tau$ is not defined on $R_{(a,1)}$ and
where $\tau$ does not have $R_{(b,1)}$ in its range.   These are not actually
as important because $H$ is small there.  In any case,
the result is that $\tau$
maps $P_l$ to $Q_l$ except for the relevant sets on the ends: $R_{(j,M)}$ or
$R_{(j,1)}$.  The error here is controlled by $N/MN  = 1/M$.  By
choosing $M$ sufficiently large, this will allow $\tau$ to be in
$W$.  But we know that $H$ is almost $\sigma$-invariant, if $N$ is large enough,
and $\epsilon$ is small enough.  Also, the choices of
the shifts of the $R_{(j,i)}$ by $\tau$ makes it so that $H - H\circ\tau$ is almost equal to $K$
in norm, by a factor controlled by $1/M$ also (since $N$ is large and $\epsilon$ is small).

The case that $p=1$ is handled similarly.  Because we use the weak*
topology, we can make approximations in the construction that would
not be possible if we were using the $L_\infty$-norm topology.

\end{proof}

\section{General cases}\label{gencases}

It is well-known that given a map $\tau$, the $\tau$-coboundaries are a set of first category.  That is, the generic function is
not a $\tau$-coboundary.  There is a dual version of this where the roles of $f$ and $\tau$ are reversed.

\begin{prop}\label{notany}  For any non-zero mean-zero $f \in L_r(X)$, with $1 \le r \le \infty$, the maps $\tau$ for which $f$ is not a $\tau$-coboundary, with a transfer function in $L_1(X)$, are a dense $G_\delta$ set.
\end{prop}
\begin{proof}  Consider the set
$\bigcup\limits_{K=1}^\infty \bigcap\limits_{n=1}^\infty \{\tau: \|S_n^\tau f\|_1 \le K\}$.  By Lin and Sine~\cite{LS}, this is the set of maps for which $f$
is a coboundary.   It is easy to see that  $\{\tau: \|S_n^\tau f\|_1 \le K\}$ is closed in the weak topology.  Indeed, for $\tau_s$ converging
weakly to $\tau$, and
any $f$, $\|S_n^{\tau_s} f\|_1 \to \|S_n^\sigma f\|_1$ as $s \to \infty$.  Hence, clearly  $\{\tau: \|S_n^\tau f\|_1 \le K\}$  and
$\bigcap\limits_{n=1}^\infty \{\tau: \|S_n^\tau f\|_r \le K\}$
are closed in the weak topology.

Now let us show that $\bigcap\limits_{n=1}^\infty \{\tau: \|S_n^\tau f\|_1 \le K\}$ has no interior.  Suppose otherwise.  Then there is a rank one map $\tau_0$
and a non-trivial weak neighborhood of $\tau_0$ contained in $\bigcap\limits_{n=1}^\infty \{\tau: \|S_n^\tau f\|_r \le K\}$.  So there exists $A_1,\dots,A_m \in \beta$
and $\epsilon > 0$ such that the weak neighborhood contains an open set $O$ of the form
$\{\tau: p(\tau A_i\Delta \tau_0 A_i) < \epsilon, i=1,\dots,m \}$.

We will be choosing $\delta_i > 0,i=1,2$ below
to give us certain estimates.
First, there is a set $H, p(H) > 0$ and $L > 0$ such that $f \ge L$ on $H$.  Let $H_0 \subset H$ with $p(H_0) = \delta_1 > 0$.
Now for any $N$ and $\delta_2 > 0$ , we can construct
a Rokhlin tower $\mathcal T = \{B_1,\dots,B_N\}$ for $\tau_0$ such that $p(\bigcup\limits_{i=1}^N B_i)  \ge 1 - \delta_2$.

Let $\sigma$ be the usual corresponding map, but take $\sigma$ to map the top level of the levels at the bottom giving $H_0$ into the first level so that $\sigma(H_0) = H_0$.  Also, take $\sigma$ to map the
top level of $\mathcal B_0$ into the level above the stack at the bottom giving $H_0$.
If $\delta_1$ and $\delta_3$ are sufficiently small (i.e. $N$ is large enough, and $\delta_i, i =1,2$ are small enough), then the resulting $\sigma$ is in $\mathcal O$ and hence in $\bigcap\limits_{n=1}^\infty \{\tau: \|S_n^\tau f\|_1 \le K\}$.

We have
$\|S_n^\sigma f\|_1 \ge \int_{H_0} |\sum\limits_{k=1}^n f\circ \sigma^k|\, dp
\ge nL\delta_1$.
But then $\|S_n^\sigma f\|_1$ can be made larger than  $2K$ if $n$ is large enough.
This contradicts $\sigma \in \bigcap\limits_{n=1}^\infty \{\tau: \|S_n^\tau f\|_1 \le K\}$.\
\end{proof}

\begin{rem}
This result gives information about the behavior of the ergodic averages $A_n^\tau f = \frac 1n\sum\limits_{k=1}^n f\circ\tau^k$.  We have $\|A_n^\tau f\|_r = O(1/n)$ exactly when
$f$ is a $\tau$-coboundary.  This gives a good rate of convergence in the mean for the ergodic averages.  Is there a (dense) class of functions which
satisfies this estimate for all $\tau$?  The answer is negative because Proposition~\ref{notany} shows that there is not even one non-zero, mean-zero function
$f \in L_r(X)$ which satisfies such an estimate (except possibly for a set first category set of maps).
\end{rem}

\begin{rem}\label{Schmidt} We should also be able to prove a generic result like Proposition~\ref{notany} where the transfer function is allowed to be just measurable.
The statement would be: Given $f \in L_1(X)$, for the generic map $\tau$, there is no measurable $h$ such that $f = h - h\circ \tau$.
To prove this, we would  use a result of K. Schmidt: $f$ is a $\tau$-coboundary if and only if for every $\epsilon>0$ , there
exists a positive real number $A$ such that for each $n\in \mathbb Z$, $|S_n^\tau f(x)|\le A$ for all $x$ in a set $E_n$
of measure at least $ 1-\epsilon$ .  See Schmidt~\cite{Schmidt}.  So the maps with this property form a subset of $\bigcup\limits_{A=1}^\infty \bigcap\limits_{n=1}^\infty \{\sigma: p\{|A_n^\sigma f| \le A\} \ge \gamma\}$
for each $\gamma > 0$.  The anticipation is that this set is first category because the closed sets $\bigcap\limits_{n=1}^\infty \{\sigma: p\{|A_n^\sigma f| \le A\} \ge \gamma\}$ have no interior for any $\gamma > 0$.  However, to date, we have not been able to prove this.
There is a similarity of this problem with the Baggett problem for rotations.  If one tries to write down the class here using a countable avatar for $\gamma$, then the natural description of the class is not a $G_\delta$ set.
\end{rem}

\begin{rem}\label{extendSchmidt} It would be worthwhile to have a result like Schmidt's Theorem, but one that applies to $f\in L_r(X)$ being a $\tau$-coboundary with transfer function in $L_s(X)$ in the cases, for example, where $0 < s < 1$ and
$s \le r$.
\end{rem}

\section {\bf Joint under powers}\label{powers}

Consider a function $f \in L_\infty(\mathbb T)$ of the form $f(\gamma) = \gamma^k$ for some fixed $k$.  Then for any rotation $\tau$ of $\mathbb T$,
say $\tau(\gamma) = \alpha\gamma$, we would have $f\circ\tau = \alpha^k f$.  So $f$ is an eigenfunction.  If $\alpha^k \not= 1$, then
for a suitable constant $c$, we have $cf = f - f\circ \tau$.  Hence, $f$ is a $\tau$-coboundary.  Thus, if $\alpha$ is infinite order and $k \not= 0$,
this shows how a function $f$ can be a common coboundary for all powers $\tau^l$ with $l \in \mathbb Z, l\not= 0$.

We can give examples of this same phenomenon with $\tau$ does not have discrete spectrum.

\begin{prop} There exists $\tau$ strongly mixing and a non-zero mean-zero $f\in L_2(\mathbb T)$ such that $f$ is a $\tau^k$-coboundary for all
$m \in \mathbb Z, m \not= 0$.
\end{prop}
\begin{proof}  For $m\ge 1$, let $\mathcal R_m$ be the $m$-th roots of unity.  Let $p$ be normalized Lebesgue measure on $\mathbb T$.  Let $V_m$ be an open neighborhood of $\mathcal R_m$ such that
$\sum\limits_{m=1}^\infty p(V_m) \le \frac 12$.  Let $K = \mathbb T\backslash \bigcup\limits_{m=1}^\infty V_m$.  Then $K$ is closed and $p(K) > 0$.
So there is a positive Borel probability measure $\nu$ on $K$ which is absolutely continuous with respect to $p$; hence, $\widehat \nu$ is in $C_0(\mathbb Z)$
by the Riemann-Lebesgue Lemma.  Using the GMC, there exists a strongly mixing transformation $\tau$ on $(X,\mathcal B,m)$ and $f_0\in L_2(X)$ such that $\widehat \nu(k) =
\langle f_0,f_0\circ\tau^k\rangle$ for all $k\in \mathbb Z$.  That is, $\nu = \nu_{f_0}^\tau$.

We will use the spectral measure $dE$ associated with the Koopman operator given by $\tau$.
We consider the function $f \in L_2(X)$ given by $f = \int_K dE(\lambda)$. Fix $m\not= 0$.
Because $K$ is disjoint from a neighborhood of $\mathcal R_m$, we can define $F\in L_2(X)$
by $\int_K \frac 1{1-\lambda^m} dE(\lambda)$.  Then $F - F\circ\tau^m = \int_K \frac 1{1-\lambda^m} dE(\lambda) - \int_K \frac 1{1 - \lambda^m}\lambda^m dE(\lambda)
= \int_K dE(\lambda) = f$.  Hence, for all $m \not= 0$, $f = F-F\circ\tau^m$ for some $F\in L_2(X)$.
\end{proof}

\section{Various constructions}\label{various}

We conjecture that we can generally write any mean-zero function as a coboundary 
(possibly with a rank-one map). 
\medskip

\noindent {\bf Conjecture}:  Given a mean-zero $f\in L_r(X), 1 \le r \le \infty$, there exists an ergodic map $\tau$ of $X$ such that $f$
is a $\tau$-coboundary with a transfer function in $L_r(X)$.
\medskip

\noindent In section \ref{cobotherway}, we show that any step function is a coboundary for some ergodic map.  A general construction is given without the need for cutting and stacking.  In the final section \ref{wmcobcutstack}, 
we use cutting and stacking to prove that any bounded measurable function, 
which is not a finite step function, is a coboundary for a weak mixing transformation. 
It is known that any transformation with a 2-step nonzero coboundary is not weak mixing.  It is not known if this is true 
for functions with 3-steps.  Our examples in section \ref{cobotherway} are not weak mixing. 

\subsection{Eigenvalues of two-step functions} Given $f$, there are some inherent restrictions on $\tau$ if $f$ is a $\tau$-coboundary.
Isaac Kornfeld pointed out that sometimes $\tau$ cannot be weakly mixing.  For example,
if $f = 1_E - p(E)$ and $f = h - h\circ \tau$, then $\exp (2\pi i f) = \exp (2\pi i p(E)) = \exp (2\pi ih)\exp (-2\pi i h
\circ \tau)$.  Hence, if $H = \exp(2\pi i h)$, we have $H\circ \tau = \exp(2\pi i p(E) )H$.  So, $\tau$ is not weakly mixing.
Indeed, if $p(E)$ is rational, then $\tau$ is not even totally ergodic.  Also, suppose $f = \sum\limits_{k=1}^K m_k1_{E_k}$
with $m_k \in \mathbb Z, k=1,\dots,K$, and
with $c = \sum\limits_{k=1}^K m_kp(E_k) \in (0,1)$ being a rational number.  Then when $f - \int f\, dp$ is a $\tau$-coboundary, we would
have $\tau$ having a non-trivial eigenvalue $\exp(2\pi i c)$ which is a root of unity, and hence again $\tau$ is not totally ergodic, and certainly not weakly mixing.


\subsection{Oxtoby-Ulam measure space}
\label{cont} Here is another type of restriction.  Consider a compact metric space $X$ and an Oxtoby-Ulam
measure on it.  Then construct an open dense set $U$ of small measure $p(U)$. Let $f = 1_U- p(U)$.  Then $f\in L_\infty(X)$ is mean-zero.
But it is easy to see that for any $p$ preserving
homeomorphism $\tau$, $\|\sum\limits_{k=1}^n f\circ T^k\|_\infty \ge n(1-p(U))$.  So $f$ is not a $\tau$-coboundary for such maps.  This can be extended to the
case where one has a compact metric space $X$ and a non-atomic probability measure $p$ on $X$.  Then one can similarly construct a mean-zero bounded
function which is not a coboundary for any $\tau$ which is a homeomorphism preserving $p$.  This shows that on
the circle $\mathbb T$ there are mean-zero bounded functions that are not coboundaries with respect to any rotation.  The arguments below give a better result via Fourier analysis.

\subsection{Fourier technique}
\label{rot}
There are also Fourier analytic versions of this, which again gives negative results.  For example, it is well-known that despite the Riemann-Lebesgue Lemma,
given any $\rho_n > 0$, with $\lim\limits_{n\to \infty} \rho_n = 0$,
there exists a mean-zero $f \in L_1(\mathbb T)$  such that $|\widehat f(n)| \ge \rho_n$ for all $n$.  Suppose we take $\rho_n = \log n/n$. Take any ergodic rotation
$\tau_\alpha$ of $\mathbb T$.   We claim that $f$ is not a $\tau_\alpha$-coboundary with a transfer function in $L_1(\mathbb T)$.  Indeed, if this were so, then
$|\widehat f(n)|/(1 -\exp(2\pi i n\alpha))$ would be bounded.  But we can choose a rational number $q = m/n$ such that $|q-\alpha|\le 1/q^2$.  Hence, $|n\alpha - m|
\le 1/n$.  But then for these values of $n$, we have  $|\widehat f(n)|/(1 -\exp(2\pi i n\alpha)) \ge \log n$, which is not bounded.

What about with continuous functions?  See de Leuuw, Katznelson, Kahane~\cite{LKK}.
They show that  given $(a_n) \in \ell_2$,
there exists a continuous function $h$ whose Fourier transform goes to zero slower than $|a_n|$.   Suppose we take $a_n = \log n/n$. Take any ergodic rotation
$\tau_\alpha$ of $\mathbb T$.   We claim that $h$ is not a $\tau_\alpha$-coboundary with a transfer function in $L_1(\mathbb T)$.  Indeed, if this were so, then
$|\widehat h(n)|/(1 -\exp(2\phi n\alpha))$ would be bounded.  then proceed as above using diophantine approximation by rational numbers.

There is an open problem in this context.  See Baggett, Medina, and Merrill and ~\cite{BMM} and Baggett~\cite{Baggett}.  The question is: given $f\in C(\mathbb T)$ which is
not a trigonometric polynomial, is the set of rotations of $\mathbb T$ for which $f$ is a coboundary with some measurable transfer function necessarily of first category?
This remains unsolved still.  However, if the transfer function is supposed to be continuous too, then this holds.  Indeed, more generally they show that $f\in L_1(\mathbb T)$
is a coboundary with transfer function in $L_1(\mathbb T)$ for at most a first category set of rotations of $\mathbb T$.
Clearly, this result bears a resemblance to the result in Proposition~\ref{notany}.

The discussion above suggests the interesting question: given a compact Hausdorff space $X$ with a non-atomic probability $p$ on $X$, is there always a mean-zero continuous
function $h$ on $X$ which is not a $\tau$-coboundary with respect to any uniquely ergodic mapping $\tau$ of $X$?

\subsection{Non-measurable solutions}
On the other hand, take any function $f$ on the integers.
Then  define another function $h$ as follows:  $h(0) = 0$, and  for
all $k\ge 1$, $h(k) = h(k-1) - f(k-1)$, and $h(-k) = h(-k+1) + f(-k)$.
Then $h$ is well defined and $f(k) = h(k) - h(k+1)$ for all $k \in \mathbb Z$.

Now take a function $F$ on a probability space $(X,\mathcal B,p)$ and an invertible map $\tau$
of $X$.  For each fixed $x$, let $f(k) = F(\tau^kx)$.  Take the associated $h$ above
and let $H(\tau^kx) = h(k)$ for all $k$.  Then we have
$F(\tau^kx) = H(\tau^kx) -H(\tau^{k+1}x)$ for all $k$.  That is, for any $y$
in the orbit $\{\tau^kx: k \in \mathbb Z\}$, we have $F(y) = H(y) - H\circ \tau(y)$.

We can repeat this construction on each orbit.  In this way, we get a function
$H$ such that $F = H - H\circ \tau$ on all of $X$.

However, to write $X$ as a disjoint union of orbits, we generally need to use the Axiom
of Choice. The resulting equation gives $H$ but $H$ might not be measurable.
For example,  if we write $1 = H - H\circ \tau$, then inherently this implies  the existence of a choice
set $E$ with one point from each orbit. Any such $E$ is not measurable if $\tau$ is measure
preserving.  It also follow that $H$ cannot be measurable.  See Anosov~\cite{Anosov} for
a general version of this phenomenon: if $f$ is integrable and a $\tau$-coboundary with respect
to a measurable transfer function, then $f$ must be mean-zero.

\section{Constructions that are not coboundaries}

It is well-known that if we fix $\tau$, then the coboundaries are a set of first category.  Here is a general version of this when the transfer function is just measurable.  See Rozhdestvenskii~\cite{Rozhdestvenskii}.

\begin{prop}\label{mostfcnsnotcob} Assume $\tau$ is ergodic.
The generic function $f \in L_r(X), 1 \le r \le \infty$ is not a $\tau$-coboundary with
a measurable transfer function.
\end{prop}
\begin{proof} We use Schmidt~\cite{Schmidt} again.
We see that the $\tau$-coboundaries are a subset of $\bigcup\limits_{A=1}^\infty\bigcap\limits_{n=1}^\infty
\{ f\in L_r(X): p\{|S_n^\tau f|\le A\} \ge 1/2\}$.  This is a countable union of $L_r$-topology closed sets
because each $\mathcal B_n(A) = \{ f\in L_r(X): p\{|S_n^\tau f| \le A\} \ge 3/4\}$ is closed in the $L_r$-topology.  We claim that
$\mathcal B(A) = \bigcap\limits_{n=1}^\infty \mathcal B_n(A)$ has no interior.  If this intersection has interior, then
there exists $f_0 \in \mathcal B(A)$ and $\delta > 0$, such that (at least) for all mean-zero $f \in L_1(X)$ with $|f| \le 1$, we
would have $f_0+\delta f \in \mathcal B(A)$.  Use the Rokhlin Lemma to construct a mean-zero function $f = \pm 1$ on a set of
measure at least $1 - \epsilon$,
such that $f\circ \tau = f$ also on a set of measure at least $1 -2\epsilon$.  This is how we can guarantee that $|S_n^\tau f| = n$
a set of measure at least $1 - 4n\epsilon$.   But
$\delta|S_n^\tau f| \le |S_n^\tau(f_0 +\delta f)| + |S_n^\tau(f_0)| \le 2A$ on a set of measure at least $1/2$.  So take first $n$ such that $n > 2A/\delta$ and then take $\epsilon$ so that $1 - 4n\epsilon > 3/4$.  Then on a set of positive measure, we have
$2A < \delta n = \delta |S_n^\tau f| \le 2A$, which is not possible.
\end{proof}

\begin{rem}\label{diffclasses} a) There are various coboundary subspaces one can consider here.  For example, there is the usual $\mathcal C_{(r,r)} = \{h-h\circ \tau: h \in L_r(X)\}$. Since $L_r(X)$ is first category in $L_s(X)$ for any $r > s \ge 0$, it is clear that these subspaces are all different.  We could also
consider $\mathcal C_{(r,s)} = \{f\in L_r(X): f = h - h\circ \tau, h \in L_s(X)\}$.  These spaces are probably all distinct.
\medskip

\noindent b) For example, if $\tau$ is ergodic, to show $C_{(1,0)}$ is different than $C_{(1,1)}$, we would need to construct $f \in L_1(X)$ such that $f$ is a $\tau$-coboundary with a measurable transfer function $h$ that is not integrable.  Here is one such construction.  Let $\mathcal T$ be a Rokhlin tower $\{E_n,\dots,\tau^{n-1}E_n\}$ for $\tau$ of height $n$ with base $E_n$,
and such that $p(E_n) = (1/2)/n$.  Take a portion $D_n \subset E_n$ with $p(D_n) = (1/2)/n^3$.  Let $h_n = \sum\limits_{k=0}^{n-1} 1_{\tau^kD_n}$.  Let
$H = \sum\limits_{n=1}^\infty n^{3/2} h_n$.  Then $H$ is a measurable function with
$\|H\|_1 = \sum\limits_{n=1}^\infty n^{3/2} (n(1/2)/n^3) = \infty$.  But also if $H_N = \sum\limits_{n=1}^N n^{3/2} h_n$, then $H_N \to H$ a.e.
because each $h_n$ is supported on $\bigcup\limits_{k=0}^{n-1} \tau^kD_n$, and $\sum\limits_{n=0}^\infty p(\bigcup\limits_{k=0}^{n-1} \tau^kD_n) = \sum\limits_{n=0}^\infty n(1/2)/n^3 < \infty$.  Also, this shows that $H$ is finite a.e.  Thus,
$f = H - H\circ \tau^{-1}$ is the limit a.e. of the coboundaries
$H_N - H_N\circ \tau^{-1} = \sum\limits_{n=1}^N n^{3/2} (1_{D_n} - 1_{\tau^nD_n})$.  We have  for all $N$,
\[|H_N - H_N\circ \tau^{-1}|
\le \sum\limits_{n=1}^N n^{3/2}(1_{D_n} + 1_{\tau^nD_n}).\]
Hence,
\[\|\sup\limits_{N \ge 1} |H_N - H_N\circ \tau^{-1}|\|_1 \le
\|\sum\limits_{n=1}^\infty n^{3/2}(1_{D_n} + 1_{\tau^nD_n})\|_1 \le \sum\limits_{n=1}^\infty 2n^{3/2}(1/2)/n^3 < \infty.\]
So, the Lebesgue Dominated Convergence Theorem
shows that  $f\in L_1(X)$.  Let $h = -H\circ\tau^{-1}$; then
 $f = h - h \circ \tau$.  Here $h$ is measurable, finite a.e., and not integrable. So $f \in C_{(1,0)}\backslash C_{(1,1)}$.
 \end{rem}

The category statement in Proposition~\ref{mostfcnsnotcob}, shows only indirectly, and not concretely, how to construct functions that are not coboundaries.  For this reason, it is worthwhile to have a better understanding of how to construct directly functions that are not coboundaries.  Here is a simple, basic example.

Suppose we construct a set $E$ such that $p(\bigcup_{k=1}^n \tau^{-k} E) < 1$ for all $n$.  Let $F = X\backslash E$.  Then consider $f = 1_F - p(F)$.  This $f$ is a bounded, mean-zero function.  For any $n$, $S_n^\tau f = n - np(F)$ on $\bigcap\limits_{k=1}^n \tau^{-k} F$.  Since
$p(\bigcup_{k=1}^n \tau^{-k} E) < 1$, we have $p(\bigcap\limits_{k=1}^n \tau^{-k} F) > 0$ for all $n$. Hence,
$\|S_n^\tau f\|_\infty \ge np(E)$ and so $f$ is not a $\tau$-coboundary with transfer function in $L_\infty(X)$.

There are many ways to construct such a set $E$.  Some of these are not explicit, which would spoil the intent of this discussion. For example,

\begin{prop}\label{setgensmall}  The generic set $E$ has $p(\bigcup_{k=1}^n \tau^{-k} E) < 1$ for all $n$
\end{prop}
\begin{proof} We consider the measurable sets $\mathcal B$ in the usual symmetry pseudo-metric.  Then $\mathcal B$ is a complete-metric space up to sets of measure zero.  Consider the class $\mathcal A$ of such sets that do not have the property above.  Let $\mathcal A_n = \{A\in\mathcal B: p(\bigcup_{k=1}^n \tau^{-k} A) = 1\}$.  Then $\mathcal A = \bigcup\limits_{n=1}^\infty \mathcal A_n$.  But clearly $\mathcal A_n$ is closed in the topology on $\mathcal B$.  It also does not have interior.  Indeed, given $A \in \mathcal A_n$ and $\epsilon > 0$, choose $E_n, p(E_n) > 0$ with $\sum\limits_{n=1}^\infty np(E_n) \le \epsilon$.
Then let $A_0 = A\backslash \left (\bigcup\limits_{n=1}^\infty \bigcup\limits_{k=1}^n \tau^kE_n\right )$.  We have $A_0 \subset A$, and
$p(A_0) \ge p(A) - \epsilon$.  But also, $\bigcup_{k=1}^n \tau^{-k} A_0 \subset X\backslash E_n$ for all $n$.  So $A_0 \notin \mathcal A_n$.
\end{proof}

The proof of Proposition~\ref{setgensmall} shows how to construct a particular set $E\notin \mathcal A$.  We just take $E = X\backslash\left (\bigcup\limits_{n=1}^\infty \bigcup\limits_{k=1}^n \tau^kE_n\right )$.  So we even can arrange that $p(E) \ge 1 -\epsilon$.  If we replace $X$ here by $A$,
then because $\epsilon$ is arbitrary, this is simple process is showing that the class of sets $\mathbb B\backslash \mathcal A$ is dense in $\mathbb B$, which of course also follows from proposition.

\begin{rem}\label{JK}  If we cite the Jewett-Krieger Theorem~\cite{Jewett,Krieger}, we can use the construction in Remark~\ref{cont} to get our explicit function that is not a coboundary.  Indeed, this theorem says that $(X,\mathcal B,p,\tau)$ is isomorphic to $(C,\mathcal B,\lambda,T)$ where $C$ is a Cantor set, $\lambda$
is Lebesgue measure, and $\tau$ is a $\lambda$-preserving, uniquely ergodic homeomorphism of $C$.  This shows that up to the isomorphism, the construction in Remark~\ref{cont} gives us bounded, mean-zero functions that are not $\tau$-coboundaries.
\end{rem}

The constructions above raise the question of how slowly we can arrange $p(\bigcup_{k=1}^n \tau^{-k} A)$ to grow.  In fact, we can get this to grow as slowly as we like.   To show this we will use this consequence of the Rokhlin Lemma.  This lemma was also an important feature in some of the arguments in del Junco and Rosenblatt~\cite{delJR}; see the corresponding lemma in this paper.

\begin{lem}\label{AI}  Suppose $\epsilon > 0$, $0 < \delta < 1$, and $n \ge 1$.  Then there is a set $A\in \mathcal B$ such that $p(A) = \delta$,
and $p\left (\bigcap\limits_{k=1}^n \tau^{-k} A \cap A\right ) \ge (1 - \epsilon)p(A)$.
\end{lem}

\begin{prop} \label{gottheslowrate} Suppose $0 <  \epsilon_n < 1/2$ and $\epsilon_n \to 0$ as $n \to \infty$.  Then there exists $E$ such that $p(\bigcup\limits_{k=1}^n \tau^{-k} E) \le 1-\epsilon_n$ for all $n$.
\end{prop}
\begin{proof} We will construct  $(N_m)$ increasing and $\delta_m > 0$ with certain properties.  First, choose $\delta_1 \ge 2\epsilon_n$ for all $n$ and so that $\gamma = 1 - 2\delta_1 >  0$.  Let $N_1 = 1$.
Choose $N_2 > 1$ sufficiently large so that if $\delta_2 = \max\limits_{n \ge N_2} 2\epsilon_n \le \gamma/4$.  Continue inductively choosing $N_{m+1} > N_m$ so that $\max\limits_{n \ge N_{m+1}} 2\epsilon_n \le \gamma/2^{2(m+1)}$ for $m \ge 2$.  Now choose $A_m$ such that $p(A_m) = \delta_m$ and
$p\left (\bigcap\limits_{k=1}^{N_{m+1}} \tau^{-k} A_m \cap A_m\right ) \ge (1/2)p(A_m)$.
Let $E = X\backslash\left (\bigcup\limits_{m=1}^\infty A_m\right )$.
We have $p(E) \ge \gamma/2$. Choose $M\ge 1$.  Then there is a unique $m\ge 1$ so that $N_m \le M < N_{m+1}$.  Also,
\[1 - p(\bigcup\limits_{k=1}^M \tau^{-k}E)\ge \bigcap\limits_{k=1}^{N_{m+1}} \tau^{-k}A_M \ge (1/2) p(A_m) \ge (1/2)\delta_m \ge \epsilon_M.\]
\end{proof}

\begin{prop}\label{estbelow}  Let $\tau$ be ergodic and let $\rho_n/n \to 0$, there exists a mean-zero $f \in L_\infty(X)$ such that $\|S_n^\tau f\|_1 \ge \rho_n$ for large enough $n$.
\end{prop}
\begin{proof}
Take $\epsilon_n\to 0$ with $\epsilon_n \ge 4\rho_n/n$ for large enough $n$, and such that $\epsilon_1 \le 3/4$.  Take $E$ as in Proposition~\ref{gottheslowrate}.  Note that $p(E) \le 3/4$.   Consider $F = X\backslash E$, and let $f = 1_F - p(F)$.  We have $p(F) \ge 1/4$.  Then $\|S_n^\tau f\|_1 \ge np(F)
p \left (\bigcup\limits_{k=1}^n \tau^{-k}(X\backslash E)\right ) \ge n\epsilon_n/4 \ge  \rho_n/$ for large enough $n$.
\end{proof}

\begin{rem}\label{estbelowconseq}  Let $\tau$ be ergodic.  Then the argument above gives an explicit, non-category construction, of a bounded mean-zero function which is not a $\tau$-coboundary with an integrable transfer function.  For example, just take $\rho_n = \sqrt n$.
\end{rem}

\begin{rem}\label{hardermethod}
Here is a more complicated argument that turns out to be more general too.
Using Lemma~\ref{AI}, we can construct $A_n$ with $p(A_n) = 1/2$ such that
\[p\left (\bigcap\limits_{k=1}^{n} \tau^{-k} A_n \cap A_n\right )
\ge \frac 12 p(A_n).\]
Consider the function $f_n= 1_{A_n} -1/2$.  We have $f_n \in L_\infty(X)$,
$f_n$ is mean-zero, and $\|f_n\|_1 = \|f_n\|_\infty= 1/2$ because $|f_n| =1/2$ on $X$.  Also,  $\|S_n^\tau f\|_1 \ge (n/2)p\left (\bigcap\limits_{k=1}^{n} \tau^{-k} A_n \right ) \ge
(n/4)p(A_n) = n/8$.  This type of unboundedness is characteristic of what is needed to construct
functions with unbounded norms for the associated cocycles.

First, we take $\epsilon_m > 0$ with $\sum\limits_{m=1}^\infty \epsilon_m \le 1$ and let $f = \sum\limits_{m=1}^\infty
\epsilon_mf_{n_m}$ for suitable $(n_m)$.  This gives $f\in L_\infty(X)$ which is mean-zero and has $\|f\|_\infty \le 1/2$.
Now, we can estimate $\|S_{n_m}^\tau f\|_1$ and arrange that this be unbounded.  As a result, $f$ is a bounded, mean-zero function which is not
a $\tau$-coboundary in any $L_r(X), 1\le r\le \infty$.  The estimate we use is
\[\|S_{n_m}^\tau f\|_1 \ge \|S_{n_m}^\tau (\epsilon_mf_m)\|_1 - \sum\limits_{k=1}^{m-1} \|S_{n_m}^\tau(\epsilon_k f_{n_k})\|_1
-\sum\limits_{k=m+1}^\infty \|S_{n_m}^\tau (\epsilon_k f_{n_k})\|_1\]
\[\ge (\epsilon_m/8)n_m -  \sum\limits_{k=1}^{m-1} \epsilon_k \|S_{n_m}^\tau f_{n_k}\|_1
-\sum\limits_{k=m+1}^\infty (\epsilon_k/2)n_m.\]
We can arrange that $\epsilon_k \to 0$ fast enough for $\sum\limits_{k=m+1}^\infty (\epsilon_k/2) \le \epsilon_m/32$.   A
quickly enough growing geometric series will work for this estimate.  But then with that sequence fixed,
we can also arrange that $n_m$ is chosen sufficiently large that
\[\sum\limits_{k=1}^{m-1} \epsilon_k \|S_{n_m}^\tau f_{n_k}\|_1
\le (\epsilon_m/32)n_m\]
by using the Mean Ergodic Theorem.  This step is not as explicit as one might like.  But now we have
$\|S_{n_m}^\tau f\|_1 \ge (\epsilon_m/16)n_m$ for all $m\ge 1$.  So for sufficiently quickly growing $(n_m)$, e.g. $n_m$ larger than
$1/\epsilon_m^2$ in addition to the other constraints on $(n_m)$, we would have $\|S_{n_m}^\tau f\|_1$ unbounded and therefore $f$ is not a $\tau$-coboundary.

 With a little more attention to the estimates, the method in Remark~\ref{hardermethod} can give a mean-zero, bounded function which is not even a $\tau$-coboundary if we
allow the use of measurable transfer functions.  The series argument used in this construction is the usual one used when one does not want to use a
 Baire category argument, like the one in Proposition~\ref{mostfcnsnotcob}, to construct examples.
\end{rem}

\section{First Category}
In this section, we prove that given any mean zero measurable function $f$ 
that is not identically zero, the set of ergodic measure preserving transformations 
$\tau$ such that $f$ is a $\tau$-coboundary with a measurable transfer function 
is of first category.  We show the collection of such $\tau$ is contained in the complement 
of a dense $G_{\delta}$ set in the weak topology. 

It is sufficient to prove this for the case in which $X = [0,1)$ is the unit interval
equipped with its Borel subsets ${\mathcal B}$, and $\mu$ is Lebesgue measure.  
Let $\Phi$ be the set of invertible measure preserving transformations 
of $([0,1), {\mathcal B}, \mu)$ endowed with the weak topology.  
Let $E_1, E_2, \ldots $ be a countable sequence of measurable sets generating
${\mathcal B}$.  Define the distance between transformations $\phi, \psi \in \Phi$ by 
\[
d(\phi,\psi) \ = \ 
\sum_{i=1}^{\infty}\frac{1}{2^i}
\left[ \, \mu ( \phi E_i \, \triangle \, \psi E_i ) 
\, + \,  \mu ( \phi^{-1} E_i \, \triangle \, \psi^{-1} E_i ) \, \right].
\]
It follows from standard results \cite{Hal56} that $(\Phi,d)$ is a complete metric space,
and that the topology generated by $d(.,.)$ coincides with the weak topology on $\Phi$. 

For each $K\in \natural$ and $\eta > 0$, define the set 
\[
G_K(f, \eta) = \{ \tau \in \Phi : \exists n\in \natural \mbox{ such that } \mu \{ x: | S_n^{\tau} f(x) | > K \} > \eta \} . 
\]

\begin{lem}
\label{lem1}
Given a nonconstant measurable function $f$, there exists $\eta > 0$ such that 
$G_K(f, \eta)$ is a dense subset of $\Phi$ for each $K\in \natural$. 
\end{lem}

\begin{proof} 
Let $\tau$ be any e.m.p.t. on $X$.  
Choose $\alpha, \beta > 0$ such that 
$\mu ( \{ x \in X : f(x) > \alpha \} ) > \beta$ and $\mu ( \{ x \in X : f(x) < - \alpha \} ) > \beta$.  
Let $B_1 = \{ x \in X : f(x) > \alpha \}$ and $B_2 = \{ x : f(x) < - \alpha \}$.  
Let $\delta > 0$.  Choose $k_1 \in \natural$ and sets $F_1, F_2 \in \Join_{i=1}^{k_1} E_i$ 
such that $\mu (B_j \triangle F_j) < \delta$ for $j = 1,2$.  Given $L > 0$, choose 
$k_1 \in \natural$ such that there exist 
$F_3, F_4 \in \Join_{i=1}^{k_2} E_i$ such that 
$\mu ( B_1 \cap F_3 ) > (1 - \delta ) \mu (F_3)$ and 
$\mu ( B_2 \cap F_4 ) > (1 - \delta ) \mu (F_4)$ and 
$\mu (F_3) > {\beta} / {(L + 1)}$ and $\mu (F_4) > {\beta} / {(L + 1)}$ and 
$\mu (F_3) < {\beta} / {L}$ and $\mu (F_4) < {\beta} / {L}$.  
Choose $N \in \natural$ such that for $n > N$ 
\[ 
\mu ( \{ x : | \frac{1}{n} \sum_{i=0}^{n-1} I_{F_3} (\tau^i x) - \mu (F_3) | > \delta \mu (F_3) \} ) < \delta 
\]
and 
\[ 
\mu ( \{ x : | \frac{1}{n} \sum_{i=0}^{n-1} I_{F_4} (\tau^i x) - \mu (F_4) | > \delta \mu (F_4) \} ) < \delta . 
\]
Let $h = {2N} / {\alpha}$.  Choose a Rohklin tower for $\tau$ of height $h$.  
Consider orbits of length $H = {N} / {\alpha}$ that begin in the bottom half of the Rohklin tower.  
We will map these points in the orbit of $\tau$ that fall in $F_3$ to points in the orbit that fall 
in $F_4$.  This produces a new e.m.p.t. $\sigma$.  
Since the sets $F_3$ and $F_4$ have small measure, $\sigma$ will be close to $\tau$ 
in the weak topology. 
Also, since $h$ is sufficiently large, enough points fall in $F_3$ and $F_4$ such that 
the resulting 
$ || S_{H}^{\sigma} f ||_1 $ will be large.  Thus, with the correct choice of the parameters, 
then $\sigma \in G_K ( f, \eta )$.  Hence, $G_K (f, \eta)$ is dense in $\Phi$. 
\end{proof}

\begin{lem}
\label{lem2}
For any nonconstant measurable function $f$, for $\eta > 0$ and $K\in \natural$, the set 
$G_K(f, \eta)$ is open in $\Phi$. 
\end{lem}

\begin{thm}
Given a mean zero, non-identically zero function $f$, 
there exists a dense $G_{\delta}$ subset $G$ of $\Phi$ such that 
$f$ is not a coboundary for each $\tau \in G$ with measurable transfer function. 
\end{thm}

\begin{proof}
Let $f$ be a nonconstant measurable function on $X$.  
By Lemma \ref{lem1}, there exists $\eta > 0$ such that $G_K(f, \eta)$ is dense for $K\in \natural$. 
By Lemma \ref{lem2}, $G_K(f, \eta)$ is open in $\Phi$.  
Thus, 
\[
G = \bigcap_{k=1}^{\infty} G_k (f, \eta) 
\]
is a dense $G_{\delta}$ subset of $\Phi$ and $f$ is not a coboundary for each $\tau \in G$ 
with a measurable transfer function. 
\end{proof}

\section{Finite Step Function Coboundaries}
\label{cobotherway}

In the following sections, we show that any bounded measurable mean-zero function
defined on $(X,\mathcal B,p)$ may be realized as a coboundary
for an ergodic map with bounded measurable
transfer function. It is well known that mean-zero step functions with 2 steps
are coboundaries for either an irrational rotation, or for a transformation
with discrete spectrum. If the base of one of the steps has measure $\lambda$,
then any transformation with this coboundary must have $\lambda$ as an eigenvalue,
and cannot be weak mixing.  See Section~\ref{various}

In this section, we extend this to show any finite step function may be realized
as a coboundary to either an ergodic translation on a torus, a discrete spectrum transformation,
or a finite extension of an ergodic translation. Also, this result may be extended
to bounded countable step functions.

In the final section, we show any bounded measurable function that
is not a countable step function may be realized as a coboundary
for a weak mixing transformation with bounded transfer function.
The weak mixing transformation is constructed iteratively 
using cutting and stacking. Also, the transfer function is constructed directly.

\begin{thm}
\label{finite_case1}
Let $(X, \B , p )$ be a Lebesgue probability space and $m$ a positive integer. Suppose $f$ is a measurable, mean-zero function with $m$ steps. In other words, the function $f$ is of the form: $f=\sum\limits_{i=1}^m a_iI_{A_i}$. If the real numbers $p (A_i)$, $i=1,2,\ldots ,m-1$ are rationally independent, then $f$ is a coboundary for an ergodic transformation isomorphic to a translation on the $(m-1)$-dimensional torus.
\end{thm}

\noindent
{\bf Proof:}
For $i=1,2,\ldots ,m$, let $\alpha_i = p (A_i)$
and $a_i = f(x)$ for $x\in A_i$.
Define $\alpha=(\alpha_1,\alpha_2,\ldots ,\alpha_m)$. Given a point
$x=(x_1,x_2,\ldots ,x_m)\in \real^m$, let
$j$ be the minimum index such that
$x_j+\alpha_j \geq x_i+\alpha_i$ for $i\neq j$,
and define $\tau_{\alpha}(x)=(y_1,y_2,\ldots ,y_m)$ where
$y_j=x_j + \alpha_j - 1$ and
$y_i = x_i + \alpha_i$ for $i\neq j$.
Let $Y$ be the closure of the orbit of the origin under $\tau_{\alpha}$.
Thus, $Y=\overline{\{\tau_{\alpha}^i(0):i\in \integers \}}$.
The set $Y$ is a closed subset of the hyperplane passing through
the origin: $\{x_1+x_2+\ldots +x_m = 1: (x_1,x_2,\ldots ,x_m)\in \real^m \}$.
The transformation $\tau_{\alpha}$ restricted to $Y$ is isomorphic to rotation
by $(\alpha_1,\alpha_2,\ldots ,\alpha_{m-1})$ on the $(m-1)$-dimensional
torus. Note, this is also isomorphic to rotation by
$(\alpha_1,\alpha_2,\ldots ,\alpha_{i-1},\alpha_{i+1},\ldots ,\alpha_m)$
for $i=1,2,\ldots ,m$.
Let $\nu$ be normalized $(m-1)$-dimensional Lebesgue measure on $Y$.

For each $j$, $1\leq j\leq m$, let
$$B_j = \{(x_i)\in Y: y_j=x_j+\alpha_j-1, (y_i)=\tau_{\alpha}(x_i) \}.$$
It is clear that $p (B_j) = \alpha_j$ for $1\leq j\leq m$.
Let $\phi:X\to Y$ be an invertible map
such that $\nu (\phi (A_i) \triangle B_i)=0$ for $1\leq i\leq m$.
The following step function $f_Y$ is a coboundary
for the ergodic transformation $\tau_{\alpha}$:
$$f_Y(x) = f(\phi^{-1}x).$$
To see this, first observe: for $(x_i) \in Y$, $x_i > -1$, and hence
$x_i < m-1$ for $1\leq i\leq m$.
Fix $x=(x_i)\in Y$ and $n\in \natural$.  Define
$$p_j = \sum\limits_{i=1}^{n} I_{B_j}(\tau_{\alpha}^i(x)).$$
The number $p_j$ corresponds to the number of times we subtract 1
when iterating with $\tau_{\alpha}$.
Thus,
$$|x_j + n\alpha_j - p_j| < m - 1$$
and hence
$$|n\alpha_j - p_j| < 2m .$$
Therefore,
\begin{eqnarray}
|\sum\limits_{i=1}^{n} f_Y(\tau_{\alpha}^ix)| &=&
|\sum\limits_{i=1}^{n} \sum\limits_{j=0}^m a_j I_{B_j}(\tau_{\alpha}^ix)| \\
&=& |\sum\limits_{j=0}^m a_j p_j| = |n\sum\limits_{j=1}^m a_j \frac{p_j}{n}| \\
&\leq& n|\sum\limits_{j=1}^m a_j (\frac{p_j}{n} - \alpha_j)| +
n|\sum\limits_{j=1}^m a_j \alpha_j | \\
&<& n\sum\limits_{j=1}^m |a_j| \frac{2m}{n} = 2m\sum\limits_{j=1}^m |a_j| .
\end{eqnarray}
By \cite{LS}, $f_Y$ is a coboundary for $\tau_{\alpha}$
with an $L_\infty(X)$ transfer function $h$.
Therefore, $h\circ \phi$ is the transfer function for coboundary $f$
and ergodic transformation $\phi^{-1}\circ \tau_{\alpha}\circ \phi$.
$\Box$

\begin{thm}
\label{finite_case2}
Suppose $f$ is a measurable finite step mean-zero function. Then $f$ is a coboundary
for an ergodic map $\tau$ in one of the following categories:
\begin{enumerate}
\item $\tau$ is a transformation with discrete spectrum;
\item $\tau$ is a product of rotations;
\item $\tau$ is a finite extension of a product of rotations.
\end{enumerate}
\end{thm}

\noindent
{\bf Proof:}
Suppose $f(x) = \sum\limits_{i=1}^{m+1} a_i I_{A_i}(x)$ where $m$ is a positive integer,
$a_i\in \real$ for $1\leq i\leq m+1$, $A_i$ are disjoint,
$p(\bigcup_{i=1}^{m+1} A_i) = 1$, and $\sum\limits_{i=1}^{m+1}a_ip(A_i) = 0$.

\noindent
{\bf Case (1):} If $p(A_i)$ is rational
for $1\leq i\leq m+1$, then $f$ is a coboundary for an ergodic map with discrete spectrum.
Let $q$ be a positive integer such that
$q p(A_i) $ is a positive integer for $1\leq i\leq m+1$. Partition each
$A_i$ into sets of width $1/q$ for $1\leq i\leq m+1$. Stack the sets
into a single column $C_0$ of height $q$.
Then cut \& stack this column repeatedly without adding spacers
to produce a rank-one transformation with discrete spectrum.
By properly assigning constant values for the transfer function
$g$ on the levels of $C_0$, we get that $f = g - g\circ \tau$.

\noindent
{\bf Case (2):} If $p(A_i)$ are rationally independent for $1\leq i \leq m$, then
Theorem \ref{finite_case1} proves that $f$ is a coboundary for an ergodic product of rotations.

\noindent
{\bf Case (3):} Suppose the set $\{ p(A_i) : 1\leq i\leq m+1 \}$ contains
a non-trivial rationally independent subset,
but $p(A_i)$ are rationally dependent for $1\leq i \leq m$.
Let $\beta_i = p(A_i)$ for $1\leq i\leq m+1$.
Suppose there exist rationals ${p_i}/{q_i}$ for $1\leq i\leq m-1$ such that
\[
\beta_m = \sum\limits_{i=1}^{m-1} \frac{p_i}{q_i} \beta_i
\]
and $\beta_i$ for $1\leq i\leq m-1$ are rationally independent.
For $1\leq i\leq m+1$, let $\alpha_i = {\beta _i}/{(1-\beta_m)}$.
Let $X_0 = (\bigcup_{i=1}^{m-1} A_i ) \cup A_{m+1}$ and define
the normalized measure $p_0 = {p} / {(1 - \beta_m)}$ on $X_0$.
By Theorem \ref{finite_case1}, the function
\begin{eqnarray}
f_{\alpha} (x) = [\sum\limits_{i=1}^{m-1} (a_i  + a_m p(A_m)) I_{A_i}(x)] + ( a_{m+1} + a_mp(A_m) ) I_{A_{m+1}}(x)
\end{eqnarray}
is a coboundary for an ergodic map on $X_0$ that is isomorphic to the rotation $R_{\alpha}:[0,1]^{m-1} \to [0,1]^{m-1}$
defined by $R_{\alpha}(x_1,\ldots ,x_{m-1}) = (x_1+\alpha_1,\ldots ,x_{m-1} + \alpha_{m-1})$ modulo one.
We will define an ergodic transformation $\tau$ as an extension of $R_{\alpha}$. Let
\[
D = \{ (x_i)\in [0,1]^{m-1} : 0 \leq \sum\limits_{i=1}^{m-1} \frac{p_i}{q_i}x_i < \beta_m \} .
\]
Define $X=\{ (x,0) : x\in [0,1]^{m-1} \} \cup \{ (x,1) : x\in D \}$. Define $\tau:X\to X$ as
\begin{eqnarray*}
\tau(x,i) =
\left\{\begin{array}{ll}
(x,1) & \mbox{if } x\in D\ \ \mbox{and}\ \ i = 0, \\
(R_{\alpha}(x), 0) & \mbox{otherwise} .
\end{array}
\right.
\end{eqnarray*}
The fact that $f$ is a coboundary for a transformation isomorphic to $\tau$ follows from the fact
that $f_{\beta}(x) = I_{D}(x) - \beta_m$ is a coboundary for $R_{\alpha}$.
One way to establish that $f_{\beta}$ is a coboundary, is to apply a generalization
of the argument found in \cite{Hec22, Pet73}. This establishes that $f_{\beta}$ is tight, and hence
a coboundary. Therefore, $f$ is a coboundary for an ergodic map
isomorphic to $\tau$.  The general case of (3) with rational dependencies can be handled
in a similar manner.
$\Box$

\noindent
{\bf Remark:} Theorem \ref{finite_case2} may be extended to bounded countable step functions.  For unbounded step functions $f=\sum_{i=0}^{\infty} a_i I_{A_i}$, the same construction can be extended in a straightforward manner to produce an ergodic measure preserving transformation $T$.  If each set $\{\mu(A_1),\mu(A_2),\ldots ,\mu(A_m)\}$ is rationally independent for each $m\in \natural$, then the method in theorem \ref{finite_case1} generates an ergodic $T$.  However, it is possible that $f$ will not be a coboundary for $T$ with measurable transfer function.  In particular, for any $r<\infty$, there exists $f\in L_r$ such that the ergodic measure preserving transformation $T$ constructed in theorem \ref{finite_case1} does not have a measurable solution $g$ to the equation $f = g - g\circ T$.

\section{Weak Mixing Coboundaries}
\label{wmcobcutstack}

\begin{thm}
\label{wmcob}
 Suppose $f\in L_{\infty}(X)$ is mean-zero, and takes on essentially infinitely many values.  There exist a weak mixing system $(X, \B , p , \tau)$ and an $L_\infty(X)$ function $g$ such that $f(x) = g(x) - g(\tau x)$ for almost every $x\in X$.
\end{thm}
The following four lemmas are the main tools for iteratively constructing our transformation and proving the previous theorem.

\subsection{Balanced Partitions}
Let $A$ be a measurable subset of $X$ and $f:A\to \real$ in $L_1(A,p_A)$.
Let $\epsilon > 0$. We say a finite partition $\Pi$ of $A$ is $\epsilon$-balanced and uniform,
if there exists $E\in \Pi$ such that:
\begin{enumerate}
\item $p (E) < \epsilon p(A)$,
\item $\int_{A\setminus E} f dp = \frac{p(A\setminus E)}{p(A)} \int_A f dp$,
\item $| f(x) - f(y) | < \epsilon$ for $x,y\in a$ and $a\in \Pi\setminus \{ E\}$,
\item $p (c) = p (d)$ for $c,d \in \Pi\setminus \{ E\}$.
\end{enumerate}
We refer to this type of partition as a PUB($\epsilon$) partition for $f_{|A}$.
The set $E$ is referred to as the exceptional set of the PUB.
\begin{lem}
Suppose $A\subset X$ is measurable and $f:A\to \real$ is integrable and takes
on essentially infinitely many values. Given $\epsilon >0$, 
there exists a PUB($\epsilon$) partition such that $f$ takes on essentially
infinitely many values on both its exceptional set $E$ and its complement
$A\setminus E$.
\end{lem}

\noindent
{\bf Proof:}
Without loss of generality, it is sufficient to prove the lemma where $0 < || f ||_{\infty} < 1$
and $\epsilon < 1$.  Let $N \in \natural$.
Choose $m\in \natural$ such that
\begin{equation}
\frac{2}{m} < \epsilon .
\end{equation}
For $i = 0,1,2,\ldots , 2m - 1$, let
\begin{equation}
A_i = \{ x\in A: -1 + \frac{i}{m} \leq f(x) < -1 + \frac{i+1}{m} \} .
\end{equation}

Let $\alpha = \min{ \{ p (A_i) : p (A_i) > 0 \} }$.
There exists $i_0$ such that $f$ takes on infinitely many values
on $A_{i_0}$. Let $E_0$ and $E_1$ be disjoint subsets of $A_{i_0}$
with equal measure and such that
\begin{eqnarray}
\frac{1}{p (E_0)} \int_{E_0} f dp < \frac{1}{p (A_{i_0})} \int_{A_{i_0}} f dp ,\\
\frac{1}{p (E_1)} \int_{E_1} f dp > \frac{1}{p (A_{i_0})} \int_{A_{i_0}} f dp ,
\end{eqnarray}
and $f$ takes on infinitely many values on the set
$A_{i_0} \setminus (E_0 \cup E_1)$ and on the set
$E_0 \cup E_1$.
Let
$$d = \min{\{ | \frac{1}{p (E_i)} \int_{E_i} f dp - \frac{1}{p (A_{i_0})} \int_{A_{i_0}} f dp | : i = 0,1 \} } .$$

By simultaneous Diophantine approximation \cite{Cas57},
there exist $q \in \natural$ and $p_i \in \natural$ such that
\begin{eqnarray}
q & > & \max{\{ \frac{2N}{(1 - \epsilon)p (A)}, \frac{2p (A)}{dp (E_1)} \}},
\end{eqnarray}
and for $i=0, 1, \ldots , 2m - 1$,
\begin{eqnarray}
| q p (A_i) - p_i | & < & q^{\frac{-1}{2m}}, \\
2m q^{{-1}/{2m}} & < & \epsilon ,\\
2mq^{{-1}/{2m}} & < & d ( \frac{2\alpha}{3} - q^{\frac{-1}{2m}} ) .
\end{eqnarray}

Let $n = q + 1$. Thus,
\begin{eqnarray}
|p (A_i) - (\frac{p_i}{n} + \frac{p (A_i)}{n})| < n^{-1}q^{{-1}/{2m}}.
\end{eqnarray}
Let $h = \sum\limits_{i=0}^{2m-1} p_i$.
For $i=0, 1, \ldots , 2m - 1$, we can choose subsets
$B_i \subset A_i$ such that
\begin{eqnarray}
p (B_i) & = & p (A_i) - \frac{p_i}{n} ,\\
\frac{1}{p (B_i)} \int_{B_i} f dp & = & \frac{1}{p (A_i)} \int_{A_i} f dp .
\end{eqnarray}
Thus,
\begin{eqnarray}
| \sum\limits_{i=0}^{2m-1} \int_{B_i} f dp | & = & | \sum\limits_{i=0}^{2m-1} \frac{p (B_i)}{p (A_i)} \int_{A_i} f dp |
= | \sum\limits_{i=0}^{2m-1} (\frac{p (B_i)}{p (A_i)} - \frac{1}{n}) \int_{A_i} f dp | \\
& \leq & \sum\limits_{i=0}^{2m-1} | p (B_i) - \frac{p (A_i)}{n} |
= \sum\limits_{i=0}^{2m-1} | p (A_i) - \frac{p_i + p (A_i)}{n} | \\
& < & 2m n^{-1} q^{{-1} / {2m}} < \frac{d}{n}
(\frac{2\alpha}{3} - q^{\frac{-1}{2m}} ) .
\end{eqnarray}
This implies we can choose $B_{i_0}$ such that
\begin{eqnarray}
\sum\limits_{i=0}^{2m-1} \int_{B_i} f dp = 0 .
\end{eqnarray}
Let $E = \bigcup_{i=0}^{2m-1} B_i$ and partition each
set $A_i\setminus B_i$ into $p_i$ subsets of measure ${1}/{n}$
to form $\Pi$.
Therefore, $p(E) < \epsilon$ and our lemma is proven. $\Box$

\subsection{Balanced Uniform Towers}
Let $A$ be a measurable subset of $X$ and $f:A\to \real$ a bounded, mean-zero function.
Given $h\in \natural$ and $\epsilon > 0$, an $\epsilon$-balanced and uniform tower for $f$ is a set
of disjoint intervals $I_i\subset A$ for $i=1,2,\ldots ,h$ and an invertible measure preserving
map $\tau: I_i \to I_{i+1}$ for $i=1,2,\ldots ,h-1$, such that:
\begin{eqnarray}
p (\bigcup_{i=1}^h I_i) & > & (1-\epsilon )p (A); \\
|\sum\limits_{i=0}^{k} f(\tau^i x) | & < & || f ||_{\infty} + \epsilon \ \ \mbox{ for } x\in I_1, k < h ,\label{nbig1}\\
\sum\limits_{i=1}^{h} \int_{I_i}f dp & = & \int_{A} f dp ,\ \mbox{and} \\
| \sum\limits_{i=0}^{h-1} f(\tau^i x) | & < & \epsilon \ \ \mbox{ for } x\in I_1 \label{nbig2}.
\end{eqnarray}
We refer to this type of tower as a TUB($\epsilon, h$) tower for $f_{|A}$.
If $f_{|A}$ has a PUB($\epsilon$), then $f_{|A}$ has a TUB($\epsilon, h$),
if there exist disjoint intervals $I_i\subset A$ for $i=1,2,\ldots ,h$
and an invertible measure preserving map
$\tau: I_i \to I_{i+1}$ for $i=1,2,\ldots ,h-1$, such that
for $i=1,2,\ldots ,h$, $I_i\in \Pi\setminus \{E\}$
and (\ref{nbig1}), (\ref{nbig2}) hold.

\begin{lem}
\label{ucl}
Let $( X, \B , p )$ be a Lebesgue probability space and $A$ a measurable subset of $X$. Suppose $f: A \to \real$ is bounded, mean-zero, and takes on essentially infinitely many values on $A$. Given $N \in \natural$ and $\epsilon > 0$, there exist $h > N$ and a TUB($\epsilon, h$) tower for $f$ such that $f$ takes on infinitely many values on both
$(\bigcup_{i=1}^{h}I_i)$ and $(A\setminus \bigcup_{i=1}^{h}I_i)$.
\end{lem}

\noindent
{\bf Proof:}
From the construction of PUB(${\epsilon} / {2}$) in the previous lemma,
partition $A_i \setminus B_i$ into a disjoint union of sets $A_i(j)$
for $j=1,2,\ldots ,p_i$, such that
\begin{eqnarray}
p (A_i(j)) & = & \frac{1}{n} .
\end{eqnarray}

\subsubsection{Greedy Stacking}
Now we give an inductive procedure for stacking the sets $A_i(j)$.
Choose arbitrary $A_i(j)$ and label the set $I_1$.
Given $I_1, I_2, \ldots , I_{k-1}$, let
\begin{equation}
\sigma_{k-1} = \sum\limits_{i=1}^{k-1} \int_{I_i} f dp .
\end{equation}
If $k = h$, then we are done.
If $\sigma_{k-1} \leq 0$,
choose
$$
I_k=A_i(j) \not\subset \bigcup_{i=1}^{k-1} \{ I_i \}
$$
such that
$\int_{I_k} f dp \geq 0$.
This is possible,
since $k < h $ and $\sigma_h = \sum\limits_{i=0}^{2m-1}\int_{A_i\setminus B_i} f dp = 0$.
Otherwise, if $\sigma_k > 0$,
then by the construction of $A_i(j)$, there
exists $I_k \not\subset \bigcup_{i=1}^{k-1} \{ I_i \}$ such that
$\int_{I_k} f dp < 0$.
This procedure produces a sequence of sets $I_i$ for $i=1, 2, \ldots , h$
with the property:
\begin{eqnarray}
\sum\limits_{i=1}^{h} \int_{I_i} f dp & = & \sum\limits_{i=0}^{2m-1} \int_{A_i\setminus B_i} f dp \\
& = & \sum\limits_{i=0}^{2m-1} \int_{B_i} f dp = 0 .
\end{eqnarray}

\subsubsection{Level Refinement}
Our transformation $\tau$ will map $I_i$ onto $I_{i+1}$ for $i = 1, 2, \ldots , h-1$.
Let $\Phi$ be the set of measure preserving maps $\tau$ such that
$I_{i+1} = \tau(I_i)$ for $i=1,2,\ldots ,h-1$.
Given $\tau\in \Phi$, disjoint subsets $D_1, D_2$ contained in $I_1$ with equal measure,
and an invertible measure preserving mapping $\psi : D_1 \to D_2$, let
\[
d(\tau, D_1,D_2, \psi) = \inf_{x\in D_1} ( \sum\limits_{i=0}^{h-1} f(\tau^ix) - \sum\limits_{i=0}^{h-1} f(\tau^i(\psi (x))) ) .
\]
Define
\[
d(D_1,D_2) = \sup_{\psi} d(\tau, D_1,D_2,\psi ) .
\]
and
\[
d(\tau) = \sup_{D_1} \{ p (D_1) | \exists D_2 \mbox{ such that } d(D_1,D_2) > \epsilon \} .
\]
Finally, let
\[
d = \inf_{\tau\in \Phi} d(\tau) .
\]
We claim that $d=0$. If $d > 0$, then there exists $\tau\in \Phi$ such that
$| d(\tau) - d | < {d} / {h}$. This produces $D_1, D_2$ and $\psi$
such that $d(\tau,D_1,D_2,\psi ) \geq \epsilon$ and $| p (D_1) - d | < {d} / {h}$.
Then there exists $0\leq i < h$ such that for $x\in D_1$,
\[
f(\tau^ix) \leq f(\tau^i(\psi x)) - \frac{\epsilon}{h} .
\]
Modify the map $\tau$, by switching $\tau^i(D_1)$ and $\tau^i(D_2)$. Thus, there exists
$\tau_1 \in \Phi$ such that $\tau_1 (\tau^{i-1}D_1) = \tau^i(D_2)$ and
$\tau_1 (\tau^{i-1}D_2) = \tau^i(D_1)$.
If $d(\tau_1, D_1,D_2,\psi) \geq \epsilon$, modify $\tau_1$ in a similar manner
to produce $\tau_2$. After a finite number of steps, we may produce
$\tau_k$ such that $d(\tau_k,D_1,D_2,\psi ) < \epsilon$.
By passing to a subset of $D_1$ if necessary, we obtain
$\tau^{\prime} \in \Phi$ such that $d(\tau^{\prime}) < d$ which proves that $d=0$ by contradiction.
Therefore, this proves (\ref{nbig2}) of our lemma. Claim (\ref{nbig1})
follows in a similar manner. $\Box$

\subsection{Weakly Balanced Uniform Towers}
Let $A$ be a measurable subset of $X$ and $f:A\to \real$ a bounded, mean-zero function.
Given $\epsilon > 0$ and $M\in \natural$, an $\epsilon$-weakly balanced and uniform tower for $f$ is a set
of disjoint intervals $I_{i,j}\subset A$ for $i=1,2,\ldots ,h_j$, $j=1,2,\ldots ,w$,
and an invertible measure preserving map $\tau: I_{i,j} \to I_{i+1,j}$ for $i=1,2,\ldots ,h_j-1$,
such that:
\begin{eqnarray}
p (\bigcup_{j=1}^{w} \bigcup_{i=1}^h I_{i,j}) & > & (1-\epsilon )p (A); \\
|\sum\limits_{i=0}^{k} f(\tau^i x) | & < & M || f ||_{\infty} \ \ \mbox{ for } x\in I_{1,j}, k < h_j ,\label{nbig3}\\
\sum\limits_{j=1}^{w} \sum\limits_{i=1}^{h_j} \int_{I_{i,j}}f dp & = & \int_{A} f dp ,\ \mbox{and} \\
| \sum\limits_{i=0}^{h_j-1} f(\tau^i x) | & < & \epsilon \ \ \mbox{ for } x\in I_{1,j} \label{nbig4}.
\end{eqnarray}
We refer to this type of tower as a W-TUB($\epsilon, M$) tower for $f_{|A}$.
If $p (I_{1,i}) = p (I_{1,j})$ for $1 \leq i,j \leq w$,
and $h_{i+1}=h_i + 1$ for $1\leq i < w$, then we say $f_{|A}$ has
a W-TUB($\epsilon, M, h_1, w$) tower.
A W-TUB($\epsilon, M, h_1, w$) tower may be derived from a PUB($\epsilon$)
in a similar manner to the way a TUB($\epsilon, h$) is obtained from a PUB($\epsilon$).

\begin{lem}[weak mixing style]
\label{wucl}
Take a measurable set $A \in \mathcal B$. Suppose $f: A \to \real$ is bounded, mean-zero, and takes on essentially infinitely many values on $A$. Given $N \in \natural$ and $\epsilon > 0$, there exist $h > N$ and a W-TUB($\epsilon, 3, h, 3$) tower for $f$ such that $f$ takes on infinitely many values on both
$(\bigcup_{j=1}^{3}\bigcup_{i=1}^{h_j}I_{i,j})$ and $(A\setminus \bigcup_{j=1}^{3} \bigcup_{i=1}^{h_j}I_{i,j})$.
\end{lem}

\noindent
{\bf Proof:}
The proof is similar to Lemma \ref{ucl}.  Since $f$ takes on essentially infinitely many values
on $A_{i_0} \setminus B_{i_0}$, there exist $d > 0$ and disjoint sets $D_1,D_2 \subset A_{i_0} \setminus B_{i_0}$
with equal positive measure such that for $x_1\in D_1$ and $x_2\in D_2$,
\[
f(x_1) < f(x_2) - d .
\]
Choose $r\in \natural$ such that
\[
r(q+1) > \frac{4 ||f||_{\infty} }{d p(D_1)} .
\]
Let $n_1 = r(q+1)$. For $i\neq i_0$, partition
$A_i \setminus B_i$ into $rp_i$ sets of measure ${1} / {n_1}$.
Choose $\lfloor \frac{4 ||f||_{\infty}}{d} \rfloor$ subsets in $D_1$
of measure ${1} / {2n_1}$, and
$\lfloor \frac{4 ||f||_{\infty}}{d} \rfloor$ subsets in $D_2$
of measure ${1} / {2n_1}$. Pair each subset from $D_1$
with a distinct subset from $D_2$. This forms
$\lfloor \frac{4 ||f||_{\infty}}{d} \rfloor$ subsets of measure
${1} / {n_1}$.  Label these subsets $A_{i_0}(j)$
for $1 \leq j \leq \lfloor \frac{4 ||f||_{\infty}}{d} \rfloor$, and
for $\lfloor \frac{4 ||f||_{\infty}}{d} \rfloor < j \leq r p_{i_0}$,
choose disjoint sets $A_{i_0}(j)$ arbitrarily of equal measure
from the remainder of $A_{i_0} \setminus B_{i_0}$.

Choose $I_1 =A_i(j) \notin A_{i_0}$. Cut $I_1$ into 3 subsets
of equal width and stack the left most third on top of the right most
third. Apply the same greedy stacking algorithm as in Lemma \ref{ucl}.
Level refinement on the sets $A_{i_0}(j)$
for $1 \leq j \leq \lfloor \frac{4 ||f||_{\infty}}{d} \rfloor$, will be able
to "tamp down" the differing values produced from the initial
cut \& stack on $I_1$.
Note $h_1 = r \sum\limits_{i=0}^{2m-1} p_i - 1$.
We leave the rest of the technical details to the reader. $\Box$

\subsection{Big Tower, Little Tower}

An extension of Lemma \ref{wucl} is needed
to construct the final transformation and transfer function.
It can be proved in the same manner as Lemma \ref{wucl}.
Given a TUB, its width is $p (I_1)$ and
given a W-TUB, its width is $\sum\limits_{j=1}^{w} p (I_{1,j})$.

\begin{lem}
\label{ucle}
Let $( X, \B , p )$ be a Lebesgue probability space and $A_1,A_2$ be disjoint measurable sets.
Suppose $f_1: A_1 \to \real$ and $f_2: A_2 \to \real$ are each bounded, measurable, mean-zero
and assume essentially infinitely many values.
Given $N \in \natural$ and $\epsilon > 0$, there exist $n > N$,
and $h_1,h_2 \in \natural$
such that each $f_j$ admit a TUB($\epsilon, h_j$) of width ${1} / {n}$,
and such that $f_j$ takes on infinitely many values on its TUB and
on the complement of its TUB.
\end{lem}
Lemma \ref{ucle} says that we can build two towers of the same width for two
different functions with disjoint supports.
The generality of the Diophantine approximation
allows the simultaneous construction of towers with the same width.
We can generalize the construction of W-TUB in the same manner.


The previous lemmas will be invoked iteratively to produce a final transformation $\tau$.
Lemma \ref{ucl} is used at the initial step. Lemma \ref{ucle} is used in the following
steps to define the induced transformation $\tau_A$ on most of the top portion of the tower
where $\tau$ has not been defined yet. Here, we point out that this may be done
in such a way that the levels of the towers approximate sets from a generating,
refining sequence of partitions. To see how to do this, in lemma \ref{ucl},
first choose sets from a fine partition that approximate
the sets $A_i$ (in lemma \ref{ucl}). Then modify these sets on a set of small measure
to produce a uniform partition of each set $A_i$. This procedure can be carried out
for both lemmas \ref{ucl} and \ref{ucle}.
This is used in the construction below to guarantee ergodicity
and subsequently weak mixing.

\subsection{Weak Mixing Construction}
For convenience, given a W-TUB($\epsilon, 3, h, 3$) with subintervals $I_{i,j}^k$, define
$I_i^k = \bigcup_{j=0}^{w-1} I_{i+j,j+1}^k$.
The following iterative technique is used to produce a weak mixing map $\tau$
and an $L_\infty$ function $g$ such that $f = g\circ \tau - g$.
Initially, Lemma \ref{wucl} is used to produce disjoint sets $X_1$ and $Y_1$ such that
$X = X_1 \cup Y_1$.
The set $X_1$ corresponds to the W-TUB produced from Lemma \ref{wucl}
and $Y_1$ its complement.
Lemma \ref{ucle} is used to decompose the part of $X_1$ where $\tau$ is not defined
into $X_1^1$ and $X_1^2$, and $Y_1$ is decomposed into $Y_1 = Y_1^1 \cup Y_1^2$.
Then $Y_1^1$ is stacked  on top of newly stacked $X_1^1$ to produce $X_2$.
Also, $Y_2 = X_1^2 \cup Y_1^2$. Once again, Lemma \ref{wucl}
is applied to decompose $X_2$ into
$X_2^1 \cup X_2^2$.  Then set $X_3 = X_2^1$ and $Y_3 = Y_2 \cup X_2^2$.

Choose sequences $\epsilon_i > 0$ and $N_i \in \natural$ for $i=0,1,\ldots$
such that
\begin{eqnarray}
\label{param1}
\sum\limits_{i=0}^{\infty} (\epsilon_i + \frac{1}{N_i}) < \infty .
\end{eqnarray}
By Lemma \ref{wucl}, there exist $n_0 > N_0$, $H_0\in \natural$,
and a W-TUB($\epsilon_0,3,H_0,3$) tower for $f$.

Let $X_1 = \bigcup_{j=1}^{3} \bigcup_{i=1}^{H_0+j-1} I_{i,j}^0$ be its W-TUB decomposition,
and $Y_1 = X\setminus X_1$.
Define $f_1^1$ on $I_{H_0 + j, j+1}^0$ for $j=0,1,2$ by
\[
f_1^1(x) = \sum\limits_{i=0}^{H_0 + j -1} f(\tau_0^{-i}x),
\]
and $f_2^1 : Y_1 \to \real$ such that $f_2^1(x) = f(x)$.
By Lemma \ref{ucle}, if we let $A_1 = \bigcup_{j=0}^{2}I_{H_0 + j, j+1}^0$
and $A_2 = Y_1$, then we generate $n_1 > N_1$ and TUBs
of width $\frac{1}{n_1}$.
There exist the following maps:
$$R_1 : I_i^1 \to I_{i+1}^1, i=1,2,\ldots ,H_1-1, I_i^1 \subset I_{H_0}^0$$
and
$$S_1 : J_i^1 \to J_{i+1}^1, i=1,2,\ldots ,h_1, J_i^1 \subset Y_1.$$
Let $\phi_1 : I_{H_1}^1 \to J_1^1$ be a measure preserving map.
Define
\begin{eqnarray*}
\tau_1(x)=
\left\{\begin{array}{ll}
\tau_0(x) & \mbox{if } x\in X_1 \setminus I_{H_0 + j, j}^0 \\
\tau_0^{-H_0+1-j} (R_1(x)) & \mbox{if } \tau^{-H_0 + 2 - j}x\in I_{i,j}^1, i<H_1 \\
S_1(x) & \mbox{if }x\in Y_1\setminus J_{h_1}^1 \\
\phi_1(x) & \mbox{if }x \in I_{H_1}^1 .
\end{array}
\right.
\end{eqnarray*}
The map $\phi_1$ may be obtained by stacking the $J_i^1$'s on top of the $I_i^1$'s.
It maps the top of the $A_1$ TUB to the bottom of the $A_2$ TUB.
The map $\tau_1$ is not defined on $J_{h_1}^1$ and on
$$Y_2 = ((\bigcup_{j=0}^{2}\bigcup_{i=1}^{H_0 + j} I_{i,j+1}^0) \cup (\bigcup_{i=1}^{H_1} I_i^1)
\cup (\bigcup_{i=1}^{h_1} J_i^1))^c .$$
Lemma \ref{wucl} is invoked with $A=J_{h_1}^1$.
We continue the construction in this manner to get the final transformation,
$$\tau(x) = \lim_{i\to \infty} \tau_i(x).$$
For convenience, let $I_{H_k+i}^k = J_i^k$ for $i=1,2,\ldots h_k$ for each $k\in \natural$.
Thus, $\tau$ may be represented by a Rokhlin tower,
$\{ I_i^k : 1\leq i \leq H_k+h_k \}$, with width ${1}/{n_k}$, and such that
$Y_{k+1} = (\bigcup_{i=0}^{H_k+h_k} I_i^k)^c$.

For each $k\in \natural$, $1\leq j\leq H_k$, and $x\in I_j^k$, define
$$g_k(x) = -\sum\limits_{i=1}^{j-1} f(\tau^{-i}x) .$$
Below we show that $\lim_{k\to \infty} g_k(x)$ exists almost everywhere.
This will imply that $g(x)=\lim_{k\to \infty} g_k(x)$ is the transfer function for $f$.

Since the accumulative measure of $J_i^k$ over $i$ for fixed $k$,
goes to zero, we can ignore this part of the space.
Suppose $x\in I_j^k$. There exists $j^{\prime}$ such that
$$g_{k+1}(x) = -\sum\limits_{i=1}^{j^{\prime}} f_1^k(R_1^{-i}\tau^{H_k+h_k-j}x) - \sum\limits_{i=1}^{j-1} f(\tau^{-i}x) .$$
By Lemma \ref{ucle},
$$|\sum\limits_{i=1}^{j^{\prime}} f_1^k(R_1^{-i}\tau^{H_k+h_k-j}x)| < ||f_1^k||_{\infty} + 3\epsilon_k .$$
Also,
$$||f_1^k||_{\infty} < 3\epsilon_{k-1} .$$
Hence,
\begin{eqnarray}
\label{cauchy}
|g_{k+1}(x) - g_k(x)| < 3(\epsilon_k + \epsilon_{k-1}) .
\end{eqnarray}
By (\ref{param1}),
$$g(x) = \lim_{k\to \infty} g_k(x)$$
exists almost everywhere.  Also, $g$ is in $L_\infty$, since $f\in L_\infty$,
and due to conditions (\ref{cauchy}) and (\ref{param1}). $\Box$

Since we are able to construct the W-TUB towers such that the levels
generate the sigma algebra, then we can apply the standard argument
of Chacon (\cite{Chacon}) to show that $\tau$ has
only 1 as an eigenvalue. Since $\tau$ is ergodic, then $\tau$ is weak mixing.
This completes the proof of Theorem \ref{wmcob}. $\Box$

\section{Unresolved Issues}

There are a number of questions that we have not been resolve.  We list them here with references to related sections in this article.

\begin{enumerate}

\item  One general question (a vague one) is this: given $f \in L_r(X), 0 \le r \le \infty$, how often is it a $\tau$-coboundary as $\tau$ varies?
\item If $f \in L_1(X)$ is mean-zero, is the set of $\tau$ such that it is a $\tau$-coboundary with a measurable transfer function necessarily of first category?  Is this also the case if $f$ is just measurable?  See Proposition~\ref{notany} and the remarks following it.
\item Is every mean-zero $f \in L_r(X), 1 \le r \le \infty$, a $\tau$-coboundary for some $\tau$ with transfer function in $L_r(X)$?
What if we allow measurable transfer functions?  Also, even in the first case, are the $\tau$ that work actually dense in $\mathcal T$ in the weak topology?  Or could it be there is only one possible map $\tau$ in some cases?  See Section~\ref{cobotherway} and Section~\ref{wmcobcutstack}.
\item Is every measurable function a $\tau$-coboundary with a measurable transfer function, for some $\tau$?
\item Given $\sigma$, is the set of $\tau$ for which there is no common coboundaries in $L_r(X)$ always of first category?  This was proved only if $\sigma$ is ergodic and rank one.  How are the results on this affected by allowing the transfer functions be just measurable?  See Section~\ref{nojoint}
\item  What is the answer to Baggett's problem?  Also, how do we characterize the case where the rotations for which the function is a coboundary are infinite in number, or even dense?  See Remark~\ref{rot}
\item What are results that distinguish classes of functions $f \in L_r(X)$ that are $\tau$-coboundaries with transfer function $h \in L_s(X)$ with $s\le r$.  Here $\tau$  could be fixed or be  allowed to vary among all of the maps.  For example, given $\tau$ which is ergodic, is there a bounded mean-zero function such that $f$ is a $\tau$-coboundary with a measurable transfer function, but not with an integrable transfer function?  See Remark~\ref{diffclasses}.
\end{enumerate}

\vskip .2in

{\bf Acknowledgements}:  We would like to thank Anush Tserunyan, Mariusz Lemanczyk, and
Isaac Kornfeld for their very helpful input while working on this article.

\medskip

{\small
\parbox[t]{5in}
{T. Adams\\
E-mail: terry@ganita.org\\}
\bigskip

{\small
\parbox[t]{5in}
{J. Rosenblatt (communicating author)\\
Department of Mathematical Sciences\\
Indiana University-Purdue University Indianapolis\\
Indianapolis, IN 46202, USA\\
E-mail: joserose@iupui.edu\\}
\bigskip

\end{document}